\documentclass[11pt]{amsart}
\usepackage[utf8]{inputenc} 	
\usepackage[T1]{fontenc}
\usepackage{lmodern}
\usepackage{yhmath}

\usepackage{amsthm,amsmath,amssymb,amsbsy,bbm,mathrsfs,supertabular,eurosym,graphicx,enumitem,xcolor}
\usepackage{mathtools}
\usepackage{leftidx} 
\usepackage{scalerel}

\newtheoremstyle{BBstyle0}  {}{}{\itshape}{}{\bfseries}{}{6pt}{}
\newtheoremstyle{BBstyle1}  {3pt}{3pt}{\rmfamily}{}{\itshape}{: }{3pt}{}
\newtheoremstyle{BBstyle2}  {3pt}{3pt}{\itshape}{}{\bfseries\large}{}{0pt}{}
\newtheoremstyle{BBstyle3}  {}{}{\itshape}{}{\bfseries}{: }{3pt}{}
\newtheoremstyle{BBstyle4}  {}{}{\rmfamily}{}{\bfseries}{}{6pt}{}
  
\usepackage[normalem]{ulem} 

\usepackage[authoryear]{natbib}
\usepackage[hidelinks]{hyperref}  

\newcommand{\pa}[1]{\left({#1}\right)}
\newcommand{\norm}[1]{\left\|{#1}\right\|}
\newcommand{\cro}[1]{\left[{#1}\right]}
\newcommand{\ab}[1]{\left|{#1}\right|}
\newcommand{\ac}[1]{\left\{{#1}\right\}}

\newcommand{\argmin}{\mathop{\rm argmin}}

\newcommand{\dfleche}[1]{\,\displaystyle{\mathop{\longrightarrow}_{#1}}\,}

\newcommand{\CV}[1]{\dfleche{#1}}

\newcommand{\E}{{\mathbb{E}}}
 
\renewcommand{\L}{{\mathbb{L}}}
\newcommand{\N}{{\mathbb{N}}}
\renewcommand{\P}{{\mathbb{P}}}
 
\newcommand{\R}{{\mathbb{R}}}
\renewcommand{\S}{{\mathbb{S}}}

\renewcommand{\H}{{\mathbb{H}}}

\newcommand{\sA}{{\mathscr{A}}}
\newcommand{\sB}{{\mathscr{B}}}
\newcommand{\sC}{{\mathscr{C}}}

\newcommand{\sL}{{\mathscr{L}}} 
\newcommand{\sM}{{\mathscr{M}}}

\newcommand{\sW}{{\mathscr{W}}}

\newcommand{\sY}{{\mathscr{Y}}}

\newcommand{\frc}{{\mathfrak{c}}}

\newcommand{\frB}{{\mathfrak{B}}}

\newcommand{\frN}{{\mathfrak{N}}}

\DeclareMathAlphabet{\mathscrbf}{OMS}{mdugm}{b}{n}

\newcommand{\sbB}{{\mathscrbf{B}}}

\newcommand{\sbM}{{\mathscrbf{M}}}

\newcommand{\sbY}{{\mathscrbf{Y}}}

\newcommand{\cA}{{\mathcal{A}}}
\newcommand{\cB}{{\mathcal{B}}}
\newcommand{\cC}{{\mathcal{C}}}

\newcommand{\cE}{{\mathcal{E}}}
\newcommand{\cF}{{\mathcal{F}}}

\newcommand{\cJ}{{\mathcal{J}}}

\newcommand{\cM}{{\mathcal{M}}}

\newcommand{\cX}{{\mathcal{X}}}
\newcommand{\cY}{{\mathcal{Y}}}

\newcommand{\gp}{{\mathbf{p}}}

\newcommand{\gs}{{\mathbf{s}}}

\newcommand{\gx}{{\mathbf{x}}}

\newcommand{\gE}{{\mathbf{E}}}

\newcommand{\gH}{{\mathbf{H}}}

\newcommand{\gP}{{\mathbf{P}}}

\newcommand{\gT}{{\mathbf{T}}}

\newcommand{\bs}[1]{\boldsymbol{#1}}

\newcommand{\bsX}{{\bs{X}}}
\newcommand{\bsY}{{\bs{Y}}} 
\newcommand{\bsZ}{{\bs{Z}}}
\newcommand{\gmu}{{\bs{\mu}}}
\newcommand{\gnu}{{\bs{\nu}}}

\newcommand{\gpi}{{\bs{\pi}}}

\newcommand{\gtheta}{{\bs{\theta}}}

\newlist{lista}{enumerate}{2}
\setlist[lista,1]{label=\alph*),ref=\alph*)}

\newlist{listi}{enumerate}{1}
\setlist[listi,1]{label=(\roman*),ref=(\roman*),align=left}

\newlist{lists}{enumerate}{3}
\setlist[lists,1]{label=$\blacktriangleright$,align=right}
\setlist[lists,2]{label=$\bullet$,align=right}
\setlist[lists,3]{label=$\star$,align=right}

\newcommand{\eref}[1]{(\ref{#1})}

\renewcommand{\ge}{\geqslant}
\renewcommand{\le}{\leqslant}
\newcommand{\1}{1\hskip-2.6pt{\rm l}}

\newcommand{\scal}[2]{\left\langle #1,#2\right\rangle}

\newcommand{\etc}[1]{#1_1,\ldots,#1_n}

\newcommand{\dps}[1]{\displaystyle{#1}}

\newcommand{\et}{^{\star}}
\newcommand{\map}[5]{
\begin{array}{l|rcl}
#1: & #2 & \longrightarrow & #3 \\
    & #4 & \longmapsto & #5 
\end{array}
    }

\newcommand{\co}[1]{\leftidx{^\mathsf{c}}{\!{#1}}{}} 

\DeclarePairedDelimiter\ceil{\lceil}{\rceil}

\newcommand{\PES}[1]{\ceil*{#1}}

\newtheorem{thm}{Theorem}
\newtheorem{lem}{Lemma}
\newtheorem{prop}{Proposition}

\newtheorem{cor}{Corollary}
\newtheorem{ass}{Assumption}

\theoremstyle{definition}
\newtheorem{exa}{Example}

\usepackage[english]{babel}

\def\ray{{r}}

\begin{document}
\title[Randomised estimators]{Statistical Inference via T-Posterior Randomised Estimators}
\author{Yannick BARAUD}
\author{Gabriel ROMON}
\address{\parbox{\linewidth}{Department of Mathematics,\\
University of Luxembourg\\
Maison du nombre\\
6 avenue de la Fonte\\
L-4364 Esch-sur-Alzette\\
Grand Duchy of Luxembourg}}
\email{yannick.baraud@uni.lu}
\date{\today}

\begin{abstract}
Given a statistical model, we propose a novel estimation method that yields randomised estimators for the unknown distribution of an observed random variable. We establish non-asymptotic bounds for the performance of these estimators and demonstrate their robustness to potential model misspecification. Notably, these properties are established by circumventing the use of concentration inequalities and empirical process theory. We provide an illustration of this approach to the problem of estimating the intensity of a Poisson process.
\end{abstract}

\maketitle

\section{Introduction}
The aim of this paper is to propose a new approach for estimating the distribution $\gP\et$ of an observed random variable $\bsX$. Given a parameterised family $\sbM=\ac{\gP_{\theta},\; \theta\in \Theta}$ of candidate probability distributions for $\gP\et$, our aim is to propose an estimation strategy that yields an estimator $\widehat \theta$ whose values belong to $\Theta$ and for which $\gP\et$ and $\gP_{\smash{\widehat \theta}}$ are as close as possible, in a sense that we shall specify later. Unlike the classical frequentist approach, in which the estimator is a deterministic function of the data $\bsX$, ours is drawn  over the parameter space $\Theta$ from a distribution $\gpi_{\bsX}$ that depends on $\bsX$. This distribution plays the same role as the posterior one in the Bayesian paradigm. From this perspective, our approach can be interpreted as a randomised estimation procedure, or a Bayes-like one, and we shall see that it offers certain advantages over the frequentist approach. 

To design our random estimator, or equivalently the distribution $\gpi_{\bsX}$ that generates it, we need two main elements.

Firstly, we need a loss function $\ell$ on the parameter space $\Theta$ as well as a test statistic $\gT(\bsX,\theta,\theta')$ between $\theta$ and $\theta'$ that allows us to compare two candidate parameters $\theta,\theta'$ in $\Theta$. When $\gP\et=\gP_{\theta\et}$ belongs to our model $\sbM$ and $\theta$ is much closer to $\theta\et$ than to $\theta'$ (relative to our loss $\ell$), we expect our test statistic $\gT(\bsX,\theta,\theta')$ to take very negative values. Since we are considering test statistics that satisfy $\gT(\bsX,\theta,\theta') = -\gT(\bsX,\theta',\theta)$ for all $(\theta,\theta')\in \Theta^{2}$, $\gT(\bsX,\theta,\theta')$ is automatically large in the opposite situation where $\theta'$ is much closer to $\theta\et$ than to $\theta$. When these parameters are either both close to or both far from $\theta\et$, the choice between one or the other becomes unimportant, as does the value of $\gT(\bsX,\theta,\theta')$. The statistic $\gT(\bsX,\theta,\theta')$ can be regarded as an estimator of the difference $\ell(\theta\et,\theta)-\ell(\theta\et,\theta')$ and we shall sometimes refer to this interpretation when describing certain heuristics.

Our second ingredient is a prior distribution $\gpi$ over the parameter space $\Theta$. This prior  (for short) will prove to be a practical tool for assessing the complexity of the parameter space, and more specifically, of an element $\theta$ within it.
 We shall use it to replace other notions of complexity such as VC-dimensions, entropies, Rademacher complexities, amongst others. Even though the latter notions are commonly used in the frequentist paradigm, they may also be difficult to evaluate in general.

To explain the heuristic that underlies our approach, let us assume for a moment that our test statistic $\gT(\bsX,\theta,\theta')$ is the ideal quantity $\ell(\theta\et,\theta)-\ell(\theta\et,\theta')$. For $\theta\in \Theta$, let us introduce 
\[
\gT(\bsX,\theta)=\int_{\Theta}\gT(\bsX,\theta,\theta')\;d\widetilde \gpi_{\bsX,\theta}(\theta')\quad \text{where}\quad \frac{d\widetilde \gpi_{\bsX,\theta}}{d\gpi}(\theta')=\frac{\exp\left(\gT(\bsX,\theta,\theta')\right)} {\int_{\Theta} \exp\left(\gT(\bsX,\theta,\theta')\right) d\gpi(\theta')}
\]
and define the distribution $\gpi_{\bsX}$ as 
%
\begin{equation}\label{def-piX-intro}
\frac{d\gpi_{\bsX}}{d\gpi}(\theta)=\frac{\exp\cro{-\gT(\bsX,\theta)}}{\int_{\Theta}\exp\cro{-\gT(\bsX,\theta)}d\gpi(\theta)}.
\end{equation}
For $\theta\in \Theta$, we note that $\widetilde \gpi_{\bsX,\theta}$ is a Gibbs measure associated with the function $\theta'\mapsto \gT(\bsX,\theta,\theta')$. Our interest in this Gibbs measure lies in the fact that it concentrates its mass around the values of $\theta'$ for which $\gT(\bsX,\theta,\theta')$ is large. Since in our case $\gT(\bsX,\theta,\theta')=\ell(\theta\et,\theta)-\ell(\theta\et,\theta')$, $\widetilde \gpi_{\bsX,\theta}$ concentrates its mass in a neighbourhood of $\theta\et$. 
Furthermore, for this specific test statistic, the density of $\widetilde \gpi_{\bsX,\theta}$ takes the form 
\[
\frac{d\widetilde \gpi_{\bsX,\theta}}{d\gpi}(\theta')=\frac{\exp\cro{-\ell(\theta\et,\theta')}}{\int_{\Theta}\exp\cro{-\ell(\theta\et,\theta')}d\gpi(\theta')}
\]
and is therefore independent of $\theta$. As a consequence, the statistic 
\begin{align*}
\gT(\bsX,\theta)=\int_{\Theta}\gT(\bsX,\theta,\theta')d\widetilde \gpi_{\bsX,\theta}(\theta')=\int_{\Theta}\pa{\ell(\theta\et,\theta)-\ell\pa{\theta\et,\theta'}}d\widetilde \gpi_{\bsX,\theta}(\theta')
\end {align*}
is of the form $\gT(\bsX,\theta)=\ell(\theta\et,\theta)+C$ for some constant $C\in\R$ and we observe that $\gpi_{\bsX}$ is also a Gibbs measure, but associated with $-\gT(\bsX,\theta)=-\ell(\theta\et,\theta)-C$ since its density takes the form
\[
\frac{d\gpi_{\bsX}}{d\gpi}(\theta)=\frac{\exp\cro{-\ell(\theta\et,\theta)-C}}{\int_{\Theta}\exp\cro{-\ell(\theta\et,\theta')-C}d\gpi(\theta')}=\frac{\exp\cro{-\ell(\theta\et,\theta)}}{\int_{\Theta}\exp\cro{-\ell(\theta\et,\theta')}d\gpi(\theta')}.
\]
The distribution $\gpi_{\bsX}$ therefore concentrates its mass around these $\theta$ for which $\ell(\theta\et,\theta)$ is small and a random variable $\widehat \theta$ with distribution $\gpi_{\bsX}$ is likely to lie close to $\theta\et$. In order to be more precise and derive a risk bound, at least in the simple situation where the parameter space $\Theta$ is finite or countable,  let us first observe that  
\begin{equation}\label{eq-intro0}
\int_{\Theta}\exp\cro{\ell(\theta\et,\theta)}d\gpi_{\bsX}(\theta)=
\frac{\gpi(\Theta)}{\int_{\Theta}\exp\cro{-\ell(\theta\et,\theta')}d\gpi(\theta')}=\frac{1}{\int_{\Theta}\exp\cro{-\ell(\theta\et,\theta')}d\gpi(\theta')}.
\end{equation}
Using \eref{eq-intro0} and the concavity property of the logarithm, we obtain that
\begin{align*}
\int_{\Theta}\ell(\theta\et,\theta)d\gpi_{\bsX}(\theta)&=\int_{\Theta}\log\exp\cro{\ell(\theta\et,\theta)}d\gpi_{\bsX}(\theta)\le \log \int_{\Theta}\exp\cro{\ell(\theta\et,\theta)}d\gpi_{\bsX}(\theta)\\
&=-\log \int_{\Theta}\exp\cro{-\ell(\theta\et,\theta)}d\gpi(\theta)=-\log\cro{ \sum_{\theta\in \Theta}\exp\cro{- \ell(\theta\et,\theta)}\gpi(\theta)}\\
&\le \inf_{\theta\in \Theta}\cro{\ell(\theta\et,\theta)+\log\pa{\frac{1}{\gpi(\theta)}}}.
\end{align*}
This bound can be interpreted as 
\[
\E\cro{\ell(\theta\et,\widehat \theta)|\bsX}\le \inf_{\theta\in \Theta}\cro{\ell(\theta\et,\theta)+\log\frac{1}{\gpi(\theta)}}
\]
when $\widehat \theta$ has a distribution $\gpi_{\bsX}$ conditionally on $\bsX$. Taking the expectation on both sides with respect to $\bsX$, we obtain the risk bound 
\begin{equation}\label{eq-intro-risk}
\E\cro{\ell(\theta\et,\widehat \theta)} \le \inf_{\theta\in \Theta} \cro{\ell(\theta\et,\theta) + \log\pa{\frac{1}{\gpi(\theta)}}}.
\end{equation}
We observe that it decomposes into two terms. The second term $\log\pa{1/\gpi(\theta)}$ depends solely on our choice of the prior distribution.
 If the parameter space $\Theta$ is finite and has cardinality $N$, we can take $\gpi(\theta)=1/N$, in which case the risk depends logarithmically on $N$ and the quantity $\log N$ can be interpreted as the complexity of our parameter space. Nevertheless, it may sometimes be more sensible to choose a prior distribution $\gpi$ that favours certain parameters over others, and $\log\pa{1/\gpi(\theta)}$ then corresponds to a complexity value associated with the parameter $\theta$. The first term on the right-hand side of \eref{eq-intro-risk} is an approximation term. It measures the quality of the approximation of $\theta\et$ by an element with complexity $\log\pa{1/\gpi(\theta)}$. The risk bound provides the best compromise between approximation and complexity amongst the elements of the parameter space.

Since $\theta\et$ is unknown, it is unfortunately impossible to take $\gT(\bsX,\theta,\theta')=\ell(\theta\et,\theta)-\ell(\theta\et,\theta')$ as we did in the heuristic argument that we presented above. The main idea is therefore to replace $\ell(\theta\et,\theta)-\ell(\theta\et,\theta')$ with a suitable estimator of it. 

In this paper, we shall see how it is possible to establish risk bounds under an appropriate control of the Laplace transform of a certain linear combination of the test statistics $\gT(\bsX,\theta,\theta')$, namely that of the random variable $\Delta_{\kappa}(\bsX,\theta,\theta',\theta'')=\kappa \gT(\bsX,\theta'' ,\theta')-\gT(\bsX,\theta'',\theta)$ where $\kappa$ denotes a positive number. Interestingly, it is sufficient to establish this control for fixed values of $\theta,\theta',\theta''$  in $\Theta$. In contrast, a frequentist approach based on minimising a criterion ${\rm crit}(\bsX,\cdot)$ over the parameter space would generally require a uniform control of the difference ${\rm crit}(\bsX,\theta)-\E\cro{{\rm crit} (\bsX,\theta)}$ over $\Theta$ to establish such risk bounds. These controls generally rely on concentration inequalities, which can be quite difficult to establish when the data are not independent. We refer the reader to Koltchinskii~\citeyearpar{Koltchinski} and Massart~\citeyearpar{MR1813803}, amongst other references, for the crucial roles that concentration inequalities play in statistical estimation and statistical learning in the frequentist paradigm. Our approach entirely circumvents these difficulties.

The test statistics we consider here have the property of yielding robust tests between two probabilities. This means that the initial assumption that $\gP\et=\gP_{\theta\et}$ belongs to $\sbM$ can be relaxed.
 The value of $\gT(\bsX,\theta,\theta')$ will in fact depend on the proximity (in a sense to be specified) of $\gP\et$ relative to $\gP_{\theta}$ and $\gP_{\theta'}$. The robustness of these statistics, combined with the fact that we only need to control the Laplace transform of $\Delta_{\kappa}(\bsX,\theta,\theta',\theta'')$, allows us to establish these risk bounds under weak assumptions on $\gP\et$, in particular not only in situations where $\bsX=(\etc{X})$ consists of independent data.

\subsection{Randomised versus non-randomised estimators}
Constructing a distribution over the parameter space from data for estimation purposes is not a new idea. 
It has been used successfully for decades in Bayesian statistics. Indeed, in the special case where $\sbM=\{\gp_{\theta}\cdot\gmu,\; \theta\in \Theta\}$ is a dominated parametric model and 
\[
\gT(\bsX,\theta,\theta')=\log\gp_{\theta'}(\bsX)-\log\gp_{\theta}(\bsX)
\]
is the difference between the log-likelihood functions, the distribution $\gpi_{\bsX}$ given by \eref{def-piX-intro} coincides with the classical Bayesian posterior distribution $\gpi_{\bsX}^{B}$. From this perspective, our approach can be regarded as a generalisation of the Bayesian approach. In this paradigm, the authors primarily study the contraction properties of the posterior distribution when the model is well-specified or, at least, when the Kullback--Leibler divergence between $\gP\et$ and a suitable element $\gP_{\theta\et}$ in the statistical model $\sbM$ is sufficiently small. We refer readers to 
Ghosal and van der Vaart~\citeyearpar{MR1790007,MR3587782}, Birg\'e~\citeyearpar{Birg__2015}, Castillo~\citeyearpar{MR4864217} amongst other references. It follows that the Bayesian approach has been shown to possess good estimation properties when the true data distribution $\gP\et=\gP_{\theta\et}$ belongs to the statistical model $\sbM$. The authors mentioned above show that, under appropriate assumptions on $\gpi$, the posterior distribution $\gpi_{\bsX}^{B}$ concentrates most of its mass on small neighbourhoods of $\theta\et$, except perhaps when $\bsX$ belongs to an unlikely set $\Omega'$ of configurations. In other words, except on a set $\Omega'$ of low probability, a random variable $\widehat \theta$ drawn from the distribution $\gpi_{\bsX}^{B}$ is close to $\theta\et$ with a probability close to 1.

Modifications of the Bayesian posterior distribution have also been proposed in the literature. We mention only a few of them. Chernozhukov and Hong~\citeyearpar{MR1984779} studied Laplace-type estimators. The results established therein are mainly asymptotic. This is not the case for Atchad\'e~\citeyearpar{MR3718168}, who studied the contraction properties of quasi-likelihoods for the estimation problem in sparse parameter spaces.

Under appropriate assumptions regarding the statistical model and $\gpi$, the Bernstein-von Mises theorem bridges the gap between the Bayesian approach and the well-known frequentist one based on maximum likelihood. Nevertheless, this result does not imply that there are no fundamental differences between the Bayesian and frequentist approaches.
 For example, if we observe $n$ i.i.d. data points drawn from a shifted density $p(\cdot-\theta\et)$ that is positive and unbounded on $\R$, the maximum likelihood estimator will not exist, whereas it is always possible to estimate the location parameter $\theta\et$ using a Bayesian approach. This simple example already illustrates the fact that there may be certain advantages to using randomised estimators rather than designing one by optimising a criterion.

Randomised estimators have also been introduced into statistical learning, notably thanks to the pioneering work of Olivier Catoni. The theory we develop here is closely related to his. In particular, Catoni was, to our knowledge, the first to introduce Gibbs estimators as an alternative to the classical empirical risk minimiser. His idea is to replace the minimisation of an empirical risk $\theta\mapsto {\rm crit}(\bsX,\theta)$ over the parameter space $\Theta$ with the random selection $\widehat \theta$ of a parameter $\theta\in \Theta$ with a Gibbs distribution 
\[
\frac{dP_{\beta}}{d\gpi}(\theta)=\frac{\exp\cro{-\beta {\rm crit}(\bsX,\theta)}}{\int_{\Theta}\exp\cro{-\beta {\rm crit}(\bsX,\theta)}d\gpi(\theta)}\quad \text{with $\beta>0$.}
\]
Catoni obtained oracle-type inequalities for $\widehat \theta$ for appropriate values of $\beta$. We refer the reader to Catoni~\citeyearpar{Catoni04}. Gibbs posterior distributions were also used by Jiang and Tanner~\citeyearpar{MR2458185} to address the problem of variable selection from a Bayesian perspective. 

Catoni also studied the risk of more general randomised estimators for which he established PAC  Bayesian bounds. His approach was used in Catoni~\citeyearpar{MR2483528} for classification purposes, in Audibert and Catoni~\citeyearpar{audibert2011linear} to estimate a regression function and some generalisations of Catoni's approach are available in Alquier~\citeyearpar{MR2483458}. We also mention that Bhattacharya et al.~\citeyearpar{MR3909926} studied the properties of the posterior distribution based on fractional likelihood. Certain PAC Bayesian  bounds were established there for the $\alpha$-R\'enyi divergence loss.
  
\subsection{Robustness}
A common undesirable feature of the Bayesian approach and some frequentist ones, typically based on the minimisation of a contrast function like the least-squares or the likelihood, is their lack of stability under misspecification. The resulting estimators may perform well when the true distribution $\gP\et$ of the data belongs to the statistical model but may also perform very poorly when this condition is not met, even when $\gP\et$ is very close to $\sbM$. The problem of robustness in statistics is a major one as statisticians often use statistical models which are approximations of reality. 

In the frequentist paradigm, this problem has been known for a long time and the search for robust estimators has led to numerous papers especially in the 1960s and 1970s. The reader may find an account of this research in Huber~\citeyearpar{MR606374}. It was tackled in Birg\'e~\citeyearpar{MR552295,MR707532,MR692785} with $T$-estimators and more recently in a series of papers based on $\rho$-estimation by Baraud et al~\citeyearpar{MR3595933}, Baraud and Birg\'e~\citeyearpar{MR3565484,BarBir2018}, Baraud and Chen~\citeyearpar{MR4725162}, Chen~\citeyearpar{MR4699568,MR4739374,MR5023015}, Sart~\citeyearpar{sart2016,MR4255173}.  

We are not aware of many examples of robust procedures in the Bayesian paradigm, at least as soon as the distance between $\gP\et$ and $\sbM$ is small enough in the Hellinger or total variation distance. The only exceptions we are aware of are Baraud and Birg\'e~~\citeyearpar{BarBir2020} and Baraud~\citeyearpar{baraud2021robust} which are, however, limited to density estimation. 

\subsection{What is new here?}
The aim of this article is to propose an alternative approach to $\rho$-estimation in order to construct optimal and robust estimators in statistics. The randomised approach we develop here offers certain theoretical and computational advantages over $\rho$-estimators. We believe they may be easier to implement, at least when the parameter space is not too complex, as they require the simulation of random variables following a given distribution $\gpi_{\bsX}$, whereas the calculation of $\rho$-estimators requires the optimisation of a criterion. As already mentioned, another motivation for these randomised estimators lies in the fact that the complexity of a parameter space equipped with a prior can be much easier to assess in the Bayesian paradigm than in the frequentist one.  This property facilitates the analysis of their risks.

The statistical method we describe here is intended to be applied across various statistical frameworks in subsequent work. In the present paper, we illustrate it by considering the problem of estimating the intensity of a Poisson process, possibly in the presence of covariates. 

Without covariates, this problem was studied by Reynaud-Bouret~\citeyearpar{MR1981635} under the $\L_{2}$-loss, and by Birg\'e~\citeyearpar{Birge-Poisson} using a Hellinger-type loss. Sart~\citeyearpar{MR3394490} extended Birg\'e's work to the setting where the intensity depends on covariates. Assuming that the target intensity is square-integrable, Reynaud-Bouret analysed an estimator defined as the minimiser of a penalised $\L_{2}$-criterion, and derived risk bounds using suitable concentration inequalities. Birg\'e~\citeyearpar{Birge-Poisson} and Sart~\citeyearpar{MR3394490} proposed an alternative approach based on $T$-estimators. These $T$-estimators are constructed from robust tests comparing the elements of a suitable discretisation of the parameter space.

A common feature of all these papers is the assumption that the observed random variable $\bsX$ is an exact realisation of a Poisson process. We do not make  this assumption here, as it is rather strong and cannot be verified in practice. Instead, we assume that $\bsX$ coincides with an unobserved Poisson process $\bsX\et$, up to the addition or removal of a small number of points. In this sense, the observed process $\bsX$ may be viewed as a corrupted version of the ideal dataset $\bsX\et$.

\subsection{Organisation of the paper and notation}
The statistical setting is presented in Section~\ref{sct-setting}, while Section~\ref{sect-def-r} introduces our measure of complexity for a parameter space -- or, more precisely, the measure of the complexity of a parameter space at a specific point. In this section, we provide several examples and establish connections with both the classical dimension of linear spaces and the entropy of more general ones. We define our posterior distribution in Section~\ref{sect-post}, where we also present our main result regarding the performance of our randomised estimator. Section~\ref{sect-Poisson} is devoted to the estimation of Poisson process intensities. The proof of our main theorem is given in Section~\ref{sect-PThm1}, whilst Section~\ref{sect-proof} contains all remaining proofs.

We now introduce our primary notation. The random variable $\bsX$ is defined on a measurable space $(\Omega,\cC)$, $\P$  is  the probability on $(\Omega,\cC)$ under which $\bsX$ has distribution $\gP\et$ and $\E$ denotes the corresponding expectation. The cardinality of a set $B$ is denoted by $|B|$ while its complement is $\co{B}$. The ceiling of a nonnegative number $x$ (the smallest integer greater than or equal to  $x$) is denoted by $\PES{x}$ and for $y\ge 0$, $x\vee y$, $x\wedge y$ are the maximum and the minimum between $x$ and $y$ respectively.  We write $\log_{+}(x)$ for $(\log x)\vee 1$ whilst $\log_{2}x$ denotes the logarithm of $x$ in base 2. Unless otherwise specified, $a/0=+\infty$ for every $a>0$. The Euclidean norm of $\R^{k}$ is denoted by $\ab{\cdot}$ without any reference to the dimension $k$. The letter $C$ denotes a positive numerical constant the value of which may change from line to line whilst $C(a)$ specifies its dependency with respect to a parameter $a$. Throughout the paper, we shall employ the functions $\psi$ and $\phi$ defined as
\begin{equation}\label{def-psi}
\map{\psi}{[0,+\infty]}{[-1,1]}{u}{
\dps{\frac{u-1}{u+1}}
}
%
%
\quad \text{and}\quad \map{\phi}{\R}{[0,+\infty)}{u}{
\dps{\frac{e^{u}-1-u}{u^{2}/2}}
}
\end{equation}
with the conventions $\psi(+\infty)=1=\phi(0)$ and $0/0=1$.
The function $\phi$ is increasing on $\R$ and we shall repeatedly use the fact that if $Z$ is a square integrable random variable not larger than $b>0$,
\begin{equation}\label{eq-debase}
\E\cro{e^{Z}}=\E\cro{1+Z+\phi(Z)\frac{Z^{2}}{2}}\le \E\cro{1+Z+\phi(b)\frac{Z^{2}}{2}}\le \exp\cro{\E(Z)+\frac{\phi(b)}{2}\E(Z^{2})}.
\end{equation}
Finally, all the measures that we consider in this paper are assumed to be positive and $\sigma$-finite.

\section{The statistical setting}\label{sct-setting}
As previously mentioned, our goal is to estimate the unknown distribution $\gP\et$ of an observed random variable $\bsX$ 
defined 
between the measurable spaces $(\Omega,\cC)$ and $(\gE,\mathcal E)$. 
To this end, we introduce certain assumptions regarding $\gP\et$ and design a parameterised statistical model $\sbM=\ac{\gP_{\theta},\; \theta\in \Theta}$ which is intended to approximate $\gP\et$. 
In the case where the model is well-specified and $\gP\et = \gP_{\theta\et}$, 
we focus on estimating the parameter $\theta\et$. To evaluate the accuracy of our estimator, we equip $\Theta$ with a loss function 
$\ell:\Theta^2 \to [0,\infty)$
acting as a distance; that is, an estimator $\widehat \theta$ is considered to perform well if and only if $\ell(\theta\et, \widehat \theta)$ is sufficiently small. In cases where the model is misspecified, $\theta\et$ should be interpreted as the parameter associated with the distribution $\gP_{\theta\et} \in \sbM$ that best approximates $\gP\et$ or provides the best compromise between approximation and complexity as in the right-hand side of \eref{eq-intro-risk}.

The estimators proposed herein are randomised, meaning they are drawn from a distribution $\gpi_{\bsX}$ on $\Theta$ that depends upon the data. We call $\gpi_{\bsX}$  the {\em $T$-posterior distribution} (or {\em $T$-posterior} for short) to emphasise its analogy with the classical Bayesian framework. The letter $T$ refers to {\em Test (statistic)}, as in the $T$-estimators introduced by Birg\'e within the frequentist paradigm.

Throughout this paper, we assume that the randomisation step can be achieved by means of a random variable $U$ (typically uniformly distributed on $[0,1]$) which is independent of $\bsX$ and, more generally, of every random variables built on $(\Omega,\cC)$ that we may consider later on. We may then see our randomised estimator $\widehat \theta=\widehat \theta(\bsX,U)$ as a deterministic function of a new observation $(\bsX,U)$ for which the conditional distribution of $\widehat \theta(\bsX,U)$ given $\bsX$ is $\gpi_{\bsX}$.

We are primarily interested in evaluating the mass of $\gpi_{\bsX}$ on sets of the form $\sB(\theta,r)=\{\theta' \in \Theta, \; \ell(\theta, \theta') \le r\}$ for $\theta \in \Theta$ and $r \ge 0$, which we call $\ell$-balls, with special interest in those centred at $\theta\et$. We treat the parameter space $\Theta$ as a measurable space by equipping it with a $\sigma$-algebra $\frB$ that ensures all such balls are measurable.

It is convenient to establish our results on a set $B$ with positive probability, typically close to 1, which is $\cC$-measurable and therefore independent of the auxiliary random variable $U$. This set corresponds to favourable configurations of $\bsX$ under which our assumptions on $\gP\et$ would hold. For example, if $\bsX$ is an $n$-tuple $(\etc{X})$ the coordinates of which are presumed to be i.i.d., $B$ might represent the set of $\omega \in \Omega$ for which the value of $\bsX(\omega)$ coincides with that of an $n$-sample built on the same measurable space $(\Omega,\cC)$. It is then of interest to examine how the risk of our estimator depends on $\P(B)$ and, consequently, on the validity of our assumption on the $X_{i}$. We denote by $\P_{B}$ and $\E_{B}$ the conditional probability and expectation given $B$. If $\widehat \theta$ is a randomised estimator with distribution $\gpi_{\bsX}$, then for every measurable subset $A$ of $\Theta$:
\begin{equation}\label{eq-cond}
\P_{B}\cro{\widehat \theta \in A} = \E_{B}\cro{\P\cro{\widehat \theta(\bsX,U) \in A | \cC}} = \E_{B}\cro{\P\cro{\widehat \theta(\bsX,U) \in A | \bsX}}=\E_{B}\cro{\gpi_{\bsX}(A)}.
\end{equation}
We shall repeatedly use this equality to link the behaviour of $\widehat \theta$ as an estimator of $\theta\et$ to the properties of $\gpi_{\bsX}$.

\section{The $\mathbf{\pi}$-complexity of the parameter space at a point}\label{sect-def-r}
Throughout this paper, we consider a prior distribution $\gpi$ on $\Theta$. As previously noted in the introduction, this distribution is used to quantify the complexity of the parameter space $\Theta$ at a point $\theta\in \Theta$.

\subsection{Definition}
Given a parameter space $\Theta$ equipped with a prior $\gpi$ and a positive number $\gamma$, we define {\em the $\gpi$-complexity of $\Theta$ at $\theta\in \Theta$} as the smallest nonnegative number $\ray(\gpi,\theta)$ which satisfies 
\begin{align}
0<\gpi\pa{\sB(\theta,2r)}\le \exp\pa{\gamma r}\gpi\pa{\sB(\theta,r)}\quad\text{for every }r\ge \ray(\gpi,\theta).
\label{prop-epsn}
\end{align}
This notion of complexity is not new and has been introduced before in the Bayesian setting for balls based on the Kullback-Leibler divergence (see for example Ghosal et al~\citeyearpar{MR1790007}). However, our definition is close to that given in Birg\'e~\citeyearpar{Birg__2015} for analysing the risk of Bayesian estimators. 

That $\ray(\gpi,\theta)$ is well-defined is the subject of Lemma~\ref{lem-r-well-defined} in Section~\ref{sect-proof}.
The quantity $\ray(\gpi,\theta)$ depends upon the value of $\gamma>0$; however, as $\gamma$ remains a numerical constant throughout our results, we shall henceforth suppress this dependency in our notation. 
We relate the value of $\ray(\gpi,\theta)$ to other classical notions of ``complexity'' in the following examples. 
\subsection{The parametric case}\label{caspara}
Let us start with the simple situation  where $\Theta$ is a convex subset of $\R^{D}$, $D\ge 1$, on which the prior $\gpi$ is equivalent to Lebesgue measure $\nu$. Also assume that the loss $\ell$ can be related to a measurable function $\frN$ on $\R^{D}$ that satisfies for some $\alpha>0$, 
\begin{equation}\label{def-homo}
\frN(av)=a^{\alpha}\frN(v)\quad \text{for every $a>0$ and $v\in\R^{D}$}.
\end{equation}
More precisely, we assume that there exist positive numbers $\underline a,\overline a, \underline b,\overline b$ such that for every measurable subset $A$ of $\Theta$ and $\theta,\theta'\in \Theta$
\begin{equation}\label{eq-Lucien}
\underline a\frN(\theta'-\theta)\le \ell(\theta,\theta')\le \overline a \frN(\theta'-\theta)\quad \text{and}\quad \underline b\nu(A)\le \gpi(A)\le \overline b\nu(A).
\end{equation}
%
The following proposition is a new version of Proposition~10 in Baraud and Birg\'e~\citeyearpar{BarBir2020} which  improves the constant and extends the result to homogeneous functions $\frN$.
\begin{prop}\label{prop-lucien}
Assume that $\Theta$ is a convex subset of $\R^{D}$ and \eref{eq-Lucien} holds for a measurable function $\frN$ satisfying \eref{def-homo}, 
where $\nu$ denotes the Lebesgue measure. Then, for every $\theta\in \Theta$
\begin{equation}\label{eq-r-para}
\ray(\gpi,\theta)\le \overline \kappa_{D} D\le \overline \kappa_{1}D\quad \text{with}\quad \overline \kappa_{D}=\frac{1}{\gamma}\cro{\frac{1}{\alpha}\log\pa{\frac{2\overline a}{\underline a}}+\frac{1}{D}\log\pa{\frac{\overline b}{\underline b}}}.
\end{equation}
\end{prop}
The proof is postponed to Section~\ref{prf-propL}. 

When the loss function is equivalent to the power of a norm on $\Theta\subset \R^{D}$ and $\gpi$ has a density with respect to the Lebesgue measure which is bounded away from 0 and infinity, the $\gpi$-complexity $\ray(\gpi,\theta)$ of $\Theta$ at $\theta$ is  bounded (independently of $\theta$) by the dimension $D$ of the ambient space $\R^{D}$. 


\subsection{A bound based on a critical radius}
Given $\theta\in \Theta$, a convenient way to control the size of the quantity $r(\gpi,\theta)$ is to assume that the prior $\gpi$ puts ``enough mass'' on an $\ell$-ball centred at $\theta$. We measure it by means of $\overline r(\gpi,\theta)$ which is the smallest nonnegative number  that satisfies $\gpi\pa{\sB(\theta,r)}\ge e^{-\gamma r}$, hence 
%
\begin{equation}\label{eq-rbar}
\overline r(\gpi,\theta)=\inf\ac{r\ge 0,\; e^{\gamma r}\gpi\pa{\sB(\theta,r)}\ge 1}.
\end{equation}
The quantity $\overline r(\gpi,\theta)$ satisfies the inequality $\gpi\pa{\sB(\theta,r)}\ge e^{-\gamma r}$ and all the values of $r\ge \overline r(\gpi,\theta)$ as well. We shall call it {\em the $\gpi$-critical radius at $\theta$}. Since for every $r\ge \overline r(\gpi,\theta)$, 
\[
\frac{\gpi\pa{\sB(\theta,2r)}}{\gpi\pa{\sB(\theta,r)}}\le \frac{1}{\gpi\pa{\sB(\theta,r)}}\le \exp\pa{\gamma r},
\]
we obtain that 
\begin{equation}\label{eq-r-rbar}
r(\gpi,\theta)\le \overline r(\gpi,\theta)\quad \text{for every $\theta\in \Theta$.}
\end{equation}
The $\gpi$-complexity of $\Theta$ at $\theta$ is therefore not larger than the $\gpi$-critical radius at $\theta$. 

\subsection{Parameter space with an entropy}
Let us now turn to the situation where for every $\epsilon>0$, $\Theta$ can be covered by a finite number of $\ell$-balls with radius $\epsilon$,
and $\ell$ is symmetric: $\ell(\theta, \theta') = \ell(\theta', \theta)$.
This implies that for every $\epsilon>0$ there exists a finite subset $\Theta[\epsilon]$ of $\Theta$ for which every $\ell$-ball centred at $\theta\in \Theta$ with radius $r\ge \epsilon$ contains at least one element of $\Theta[\epsilon]$. If $\gpi=\gpi_{\epsilon}$ is the uniform distribution on $\Theta[\epsilon]$, we derive that
\[
\gpi_{\epsilon}\pa{\sB(\theta,r)}\ge \frac{1}{\ab{\Theta[\epsilon]}}\quad \text{for every $\theta\in \Theta$ and $r\ge \epsilon$}.
\]
This implies that 
\begin{equation}\label{eq-rgamma}
r(\gpi_{\epsilon},\theta)\le \overline r(\gpi_{\epsilon},\theta)\le \max\ac{\epsilon,\frac{\log\ab{\Theta[\epsilon]}}{\gamma}}\quad \text{for every $\theta\in \Theta$.}
\end{equation}
Assuming, with no loss of generality, that $\epsilon\mapsto \ab{\Theta[\epsilon]}$ is non-increasing on $(0,+\infty)$, we may minimise the right-hand side of this inequality by choosing $\epsilon\et=\inf\{\epsilon>0,\; \log\ab{\Theta[\epsilon]}\le \gamma\epsilon\}$. In this case, $r(\gpi_{\epsilon\et},\theta)\le \epsilon\et$. 

\subsection{Hierarchical priors}
Let us now consider the situation where $\Theta=\bigcup_{m\in\cM}\Theta_{m}$ is an at most countable union of parameter spaces $\Theta_{m}$. This framework typically arises when we wish to select a suitable model for $\theta\et$ from a collection $\{\Theta_{m},\; m\in\cM\}$ of candidates. We therefore assume that each model $\Theta_{m}$ is equipped with its own prior distribution $\gpi_{m}$, and we additionally consider a prior on the index set $\cM$ that serves to favour certain models over others. Since $\cM$ is assumed to be at most countable, it suffices to assign a probability of the form $\exp(-L_{m})$ to each $m\in \cM$. Here, $L_{m}$ may be interpreted as a nonnegative weight associated with the model $\Theta_{m}$ with the additional constraint that
\[
\sum_{m\in\cM}e^{-L_{m}}=1.
\]
This results in a prior $\gpi$ on $\Theta$ defined by $\gpi=\sum_{m\in\cM}e^{-L_{m}}\gpi_{m}$ that is called {\em a hierarchical prior}. We note that for every $m\in\cM$ and $r>0$, 
\begin{align*}
\gpi\pa{\sB(\theta,r)}&=\sum_{m'\in\cM}e^{-L_{m'}}\gpi_{m'}\pa{\sB(\theta,r)}\ge e^{-L_{m}}\gpi_{m}\pa{\sB(\theta,r)}.
\end{align*}
In particular, for $r\ge \overline r(\gpi_{m},\theta)+L_{m}/\gamma\ge  \overline r(\gpi_{m},\theta)$ 
\begin{align*}
\gpi\pa{\sB(\theta,r)}\ge e^{-L_{m}}\gpi_{m}\pa{\sB(\theta,\overline r(\gpi_{m},\theta))}\ge \exp\cro{-\gamma\pa{\overline r(\gpi_{m},\theta)+\frac{L_{m}}{\gamma}}}\ge \exp\pa{-\gamma r}.
\end{align*}
Since the index $m$ can be chosen arbitrarily in $\cM$, we obtain that 
\begin{equation}\label{eq-r-Hier}
r(\gpi,\theta)\le \overline r(\gpi,\theta)\le \inf_{m\in\cM}\cro{\overline r(\gpi_{m},\theta)+\frac{L_{m}}{\gamma}}.
\end{equation}

\section{The posterior distribution and its main properties}\label{sect-post}
\subsection{The posterior distribution}
Our estimation strategy is based on the choice of a prior $\gpi$ on $\Theta$ and a suitable test statistic $\gT(\bsX,\theta,\theta')$ that allows one to compare two parameters $\theta\neq \theta'$ in $\Theta$. Given two numbers $\lambda>0$ and $\beta\in (0,1)$, we define  the density with respect to $\gpi$ of our posterior distribution $\gpi_{\bsX}$ on $\Theta$ by 
\begin{equation}\label{def-piX}
\frac{d\gpi_{\bsX}}{d\gpi}(\theta)=\frac{\exp\cro{-\lambda(1-\beta)\gT(\bsX,\theta)}}{\int_{\Theta}\exp\cro{-\lambda(1-\beta)\gT(\bsX,\theta)}d\gpi(\theta)}\quad \text{for $\theta\in \Theta$}
\end{equation}
where 
%
\begin{align}
\gT(\bsX,\theta)
&=\int_{\Theta}\gT(\bsX,\theta,\theta')\frac{\exp\cro{\lambda\gT(\bsX,\theta,\theta')}}{\int_{\Theta}\exp\cro{ \lambda \gT(\bsX,\theta,\theta')}d\gpi(\theta')}d\gpi(\theta').\label{def-TX}
\end{align}
Heuristically, this posterior distribution puts most of its mass around parameters $\theta$ which minimise over $\Theta$ the mapping $\theta\mapsto \sup_{\theta'\in \Theta}\gT(\bsX,\theta,\theta')$. By using Gibbs-like density measures, our randomised estimator mimics the estimator we could get in the frequentist paradigm by considering
\[
\argmin_{\theta\in \Theta}\sup_{\theta'\in \Theta}\gT(\bsX,\theta,\theta').
\]

The results established in this paper hold as soon as the quantities which appear in \eref{def-piX} and \eref{def-TX} are well-defined.  This means that the integrals involved in these definitions are finite and that $\gT(\bsX,\theta,\theta')$ is a measurable function of $\bsX,\theta$ and $\theta'$.  These conditions are satisfied under the assumption below, which is met in all the examples considered later.
\begin{ass}
\label{Ass-posterior-well-defined}
There exist two measurable functions
$\gT_0:\gE\times \Theta^2 \to \R$, $f:\gE\times \Theta \to [0,+\infty)$ such that for each $\gx\in \gE$:
\begin{enumerate}
  \item $\gT(\gx,\theta,\theta') = \gT_0(\gx,\theta,\theta') + f(\gx,\theta) - f(\gx,\theta')$ holds for every $(\theta,\theta')\in \Theta^2$,
  \item There exists $K_\gx\in [0,\infty)$ such that $|\gT_0(\gx,\theta,\theta')|\le K_\gx$ for every $(\theta,\theta')\in \Theta^2$.
\end{enumerate}
\end{ass}
Under this assumption, our posterior is well-defined, as shown in Lemma~\ref{lem-posterior-well-defined} of Section~\ref{sect-proof} 
and the mapping  $(\bsX,\theta)\mapsto d\gpi_{\bsX}(\theta)/d\gpi$ is a measurable function of $(\bsX,\theta)$.

\subsection{The main assumption}
For $\kappa>0$ and $\gtheta=(\theta_{1},\theta_{2},\theta_{3})$ in $\Theta^{3}$, we denote by $\Delta_{\kappa}(\bsX,\gtheta)$ the random variable 
\begin{align}
\Delta_{\kappa}(\bsX,\gtheta)=\kappa\gT(\bsX,\theta_{3},\theta_{2})-\gT(\bsX,\theta_{3},\theta_{1})\label{Delta-k}
\end{align}
and by $\sL_{\kappa}(\gtheta,\lambda|B)$ its Laplace transform at $\lambda>0$ conditionally on a set $B$ with positive probability. That is, 
\[
\sL_{\kappa}(\gtheta,\lambda|B)=\E_{B}\cro{\exp\pa{\lambda\Delta_{\kappa}(\bsX,\gtheta)}}=\frac{1}{\P(B)}\E\cro{\exp\pa{\lambda\Delta_{\kappa}(\bsX,\gtheta)}\1_{B}}.
\]
The following condition links the test statistic $\gT(\bsX,\cdot,\cdot)$ to the loss $\ell$. 
%
\begin{ass}\label{Ass-CasDet}
Let $\lambda>0$, $\beta\in (0,1)$, $\overline \beta=2-\beta\in (1,2)$ and 
$\frc_{1}(\lambda,\beta)$,
$\frc_{1}(\lambda,\overline \beta)$,
$\frc_{2}(\lambda,\beta)$,
$\frc_{2}(\lambda,\overline \beta)$,
$\frc_{3}(\lambda,\beta)$,
$\frc_{3}(\lambda,\overline \beta)$ be numbers depending on $\lambda$ and $\beta$ that satisfy the constraints 
\begin{equation}\label{cond-ei}
\frc_{1}(\lambda,\beta)\wedge \frc_{1}(\lambda,\overline \beta)\ge 0,\quad \frc_{2}(\lambda,\beta)\wedge \frc_{2}(\lambda,\overline \beta)>0,\quad \text{$\frc_{3}(\lambda,\overline \beta)\ge 0$ and $\frc_{3}(\lambda,\beta)<0$}.
\end{equation}
There exist $\theta\et\in \Theta$, an event $B$ with $\P(B)>0$ and a nonnegative number $A_{0}(\lambda,\kappa,\gP\et,B,\theta\et)$ such that for every $\gtheta=(\theta_{1},\theta_{2},\theta_{3})\in \Theta^{3}$ and $\kappa\in\{\beta,\overline \beta\}$, 
\begin{align}
\sL_{\kappa}(\gtheta,\lambda|B)\le  \exp\cro{A_{0}(\lambda,\kappa,\gP\et,B,\theta\et)+\sum_{i=1}^{3}(-1)^{i-1}\frc_{i}(\lambda,\kappa)\ell\pa{\theta\et,\theta_{i}}}.\label{BLaplace0}
\end{align}
\end{ass}


In our applications, $\lambda$ and $\beta$ are numerical constants that are tuned to satisfy the constraints given by \eref{cond-ei}. Given the values of $\lambda$ and $\beta$, the quantities $\frc_{1}(\lambda,\kappa),\frc_{2}(\lambda,\kappa)$ and $\frc_{3}(\lambda,\kappa)$ with $\kappa\in\{\beta,\overline \beta\}$ are therefore numerical constants as well. Even though we only require that \eref{BLaplace0} be satisfied for some parameter $\theta\et\in \Theta$, this inequality turns out to hold for every $\theta\et\in \Theta$ in our applications. The set $B$ may depend on $\theta\et$.

We provide a heuristic explanation below to offer a better insight into our Assumption~\ref{Ass-CasDet}. Let us first assume that $\gP\et=\gP_{\theta\et}$ belongs to $\sbM$ and $B=\Omega$ for the sake of simplicity. For small enough values of $\lambda$, the Laplace transform of $\Delta_{\kappa}(\bsX,\gtheta)$ at $\lambda$ is then of order 
\begin{align*}
\E\cro{1+\lambda \Delta_{\kappa}(\bsX,\gtheta)}&=1+\lambda\pa{\kappa\E\cro{\gT(\bsX,\theta_{3},\theta_{2})}-\E\cro{\gT(\bsX,\theta_{3},\theta_{1})}}.
\end{align*}
If $\gT(\bsX,\theta,\theta')$ were an unbiased estimator of $\ell(\theta\et,\theta)-\ell(\theta\et,\theta')$, we would then get 
\begin{align*}
\E\cro{1+\lambda \Delta_{\kappa}(\bsX,\gtheta)}&=1+\lambda\cro{\kappa\pa{\ell(\theta\et,\theta_{3})-\ell(\theta\et,\theta_{2})}-\pa{\ell(\theta\et,\theta_{3})-\ell(\theta\et,\theta_{1})}}\\
&=1+\lambda\ell\pa{\theta\et,\theta_{1}}-\lambda \kappa \ell\pa{\theta\et,\theta_{2}}+\lambda\pa{\kappa -1}\ell\pa{\theta\et,\theta_{3}}\\
&\le \exp\cro{\lambda\ell\pa{\theta\et,\theta_{1}}-\lambda \kappa \ell\pa{\theta\et,\theta_{2}}+\lambda\pa{\kappa -1}\ell\pa{\theta\et,\theta_{3}}}.
\end{align*}
As a consequence, at least for these small enough values of $\lambda$, our Assumption~\ref{Ass-CasDet} on the Laplace transform of $\Delta_{\kappa}(\bsX,\gtheta)$ is satisfied with $A_{0}(\lambda,\kappa,\gP_{\theta\et}, \Omega,\theta\et)=0$, $\frc_{1}(\lambda,\kappa)=\lambda$, $\frc_{2}(\lambda,\kappa)=\lambda \kappa$ and $\frc_{3}(\lambda,\kappa)=\lambda\pa{\kappa -1}$ for $\kappa\in\{\beta,\overline \beta\}$. Note that the constraints in \eref{cond-ei} are then automatically  satisfied since $\beta\in (0,1)$ and $\overline \beta=2-\beta\in (1,2) $. In fact, Assumption~\ref{Ass-CasDet} would also be met for $\lambda$ close enough to 0 even if $\gT(\bsX,\theta,\theta')$ were not an unbiased estimator of $\ell(\theta\et,\theta)-\ell(\theta\et,\theta')$ provided that it satisfied for every $\theta,\theta'\in \Theta$
\[
c\ell(\theta\et,\theta)-C\ell(\theta\et,\theta')\le \E\cro{\gT(\bsX,\theta,\theta')}\le C\ell(\theta\et,\theta)-c\ell(\theta\et,\theta')
\]
for constants $0<c\le C$. 

When $\gP\et$ does not belong to $\sbM$ and $B\ne \Omega$, we shall see in our examples that the logarithm of $\sL_{\kappa}(\gtheta,\lambda|B)$ inflates by some quantity $A_{0}(\lambda,\kappa,\gP\et,B,\theta\et)>0$. It measures  the impact of a possible misspecification of the model and the fact that our observation $\bsX$ may take unfavourable configurations.   

In view of providing non-asymptotic results on the performance of our randomised estimator, we wish to make the numerical constants which are involved in our main results as explicit as possible. Hereafter, we provide a list of two constants which can be calculated from the values of $\lambda$ and $\beta$:
\begin{align}
\gamma=\gamma(\lambda,\beta)&=\frac{1}{3}\pa{\frc_{2}(\lambda,\beta)\wedge \frc_{2}(\lambda,\overline \beta)\wedge |\frc_{3}(\lambda,\beta)|}>0\label{def-gamma};\\
J=J(\lambda,\beta)&=2+\PES{\log_{2}\pa{1+\frac{\frc_{1}(\lambda,\overline \beta)+\frc_{1}(\lambda,\beta)+\frc_{3}(\lambda,\overline \beta)}{2\gamma}}}.\label{def-J}
\end{align}
These constants are probably not sharp and only aim at making our statement more precise. The value of $\gamma$ that results from these calculations is that which is used throughout this paper in the definitions of the quantities $r(\gpi,\theta)$ and $\overline r(\gpi,\theta)$ that are given in Section~\ref{sect-def-r}.

We also use the notation
\begin{equation}\label{def-e-r}
A_{0}(\gP\et,B,\theta\et)=\frac{A_{0}(\lambda,\beta,\gP\et,B,\theta\et)+A_{0}(\lambda,\overline \beta, \gP\et,B,\theta\et)}{2},
\end{equation}
hence dropping the dependency with respect to the numerical constants $\lambda$ and $\beta$.

We illustrate our assumption in the density framework.
\begin{exa}[The density framework]
Let $\bsX = (X_{1}, \ldots, X_{n})$ be an $n$-tuple of random variables taking values in a measurable space $(\cX, \cA)$. 
One may naturally model these variables as i.i.d.\ with a distribution belonging to a parametrised model $\{P_{\theta} = p_{\theta} \cdot \mu, \;\theta \in \Theta\}$ dominated by a measure $\mu$, 
although this model may be misspecified.
We wish to analyse the performance of our random estimator in a scenario where reality deviates from our model. The observed data $\{X_{1}, \ldots, X_{n}\}$ may be a corrupted version of an ideal dataset $\{X_{1}\et, \ldots, X_{n}\et\}$ consisting of genuine independent random variables on $(\Omega,\cC)$. In particular, $\{X_{1}, \ldots, X_{n}\}$ may contain a small proportion of dependent data, some of which correspond to repetitions of $X_{i}\et$ or a function thereof.

We denote by $I$ the set of indices for which $X_{i}$ and $X_{i}\et$ do not coincide, that is
\begin{equation}\label{def-I}
I = \left\{i \in \{1, \ldots, n\} \colon X_{i} \neq X_{i}\et\right\}.
\end{equation}
This set $I$ may be random, and we denote by $t_{\xi}$ the $(1 - e^{-\xi})$-quantile of $|I|$. The marginal distribution of $X_{i}\et$ for $i \in \{1, \ldots, n\}$ is denoted by $P_{i}\et$, and we denote their average by $P\et = n^{-1} \sum_{i=1}^{n} P_{i}\et$. We expect that our statistical model provides a reasonable approximation to most of the marginals $P_{1}\et, \ldots, P_{n}\et$, and more specifically to their average $P\et$, with respect to the Hellinger distance.

We recall that the Hellinger distance $h$ between two probability measures $P = p \cdot \nu$ and $Q = q \cdot \nu$ on $(\cX, \cA)$, both dominated by $\nu$, is
\[
h(P, Q) = \sqrt{\frac{1}{2} \int_{\cX} \left(\sqrt{p} - \sqrt{q}\right)^{2} \, d\nu}
\]
the result being independent of the choice of $\nu$. For this problem, the test statistic is
\begin{equation}\label{eq-T1}
\gT_{1}(\bsX,\theta,\theta')=\sum_{i=1}^{n}\psi\pa{\sqrt{\frac{p_{\theta'}(X_{i})}{p_{\theta}(X_{i})}}}\quad \text{for $\theta,\theta'\in \Theta$},
\end{equation}
where $\psi$ was defined in \eref{def-psi}, and we recall that $\psi$ takes values in $[-1,1]$.
We assume additionally that the mapping $(x,\theta)\mapsto p_\theta(x)$ is measurable, so that $\gT_{1}$ satisfies Assumption~\ref{Ass-posterior-well-defined}. 
Next, we show that Assumption~\ref{Ass-CasDet} holds.

\begin{prop}\label{prop-densite}
Let $\xi>0$ and $\theta\et$ be an arbitrary point in $\Theta$. Assumption~\ref{Ass-CasDet} is satisfied at $\theta\et$ with $\gT=\gT_{1}$, $B= \{|I|\le t_{\xi}\}$, 
$\ell(\theta,\theta')=nh^{2}(P_{\theta},P_{\theta'})$ for $\theta,\theta'\in \Theta$ provided that  $(\lambda,\beta)\in (0,+\infty)\times (0,1)$ satisfies the constraints  
\begin{equation}\label{eq-constraints1}
\lambda\overline \beta \phi(\lambda (1+\overline \beta))<\frac{1}{4}\quad \text{and}\quad  
\frac{5\beta}{4}+\lambda(\beta^{2}+1)\phi( \lambda (1+\beta))<\frac{1}{4},
\end{equation}
with $\phi$ defined in \eref{def-psi}. The values $\lambda=0.11$ and $\beta=0.09$ suit, in which case one may take $\gamma=7.9\times 10^{-4}$ and $J=13$. The expressions of the numerical constants $\frc_{i}(\lambda,\kappa)$ for $i\in\{0,1,2,3\}$ and $\kappa\in\{\beta,\overline \beta\}$ are provided in the proof of this proposition in Section~\ref{prf-prop-densite} and $A_{0}(\lambda,\kappa,\gP\et,B,\theta\et)=2\lambda (1+\kappa)t_{\xi}-\log(1-e^{-\xi})+\frc_{0}(\lambda,\kappa)nh^{2}(P\et,P_{\theta\et})$. 
\end{prop}
\end{exa}
The density framework was studied in Baraud~\citeyearpar{baraud2021robust} (with $\psi/2$ in place of $\psi$) under the assumption that $X_{i}=X_{i}\et$ for all $i\in\{1,\ldots,n\}$, hence, under the assumption that the $X_{i}$ are truly independent.

\subsection{The main result}
The aim of this section is to provide a bound on the loss $\ell(\theta\et,\widehat \theta)$ when $\widehat \theta$ has distribution $\gpi_{\bsX}$. Interestingly, this bound can directly be inferred from the control of the Laplace transform of the random variable $\Delta_{\kappa}(\bsX,\gtheta)$ given by \eref{Delta-k}. 
\begin{thm}\label{thmCasDet}
Let $\theta\et\in \Theta$ and $\xi>0$. Assume that $\lambda>0$, $\beta\in (0,1)$ and the event $B$ satisfy Assumption~\ref{Ass-CasDet} at $\theta\et$. Then our posterior $\gpi_{\bsX}$ defined by~\eref{def-piX} satisfies, 
\begin{equation}\label{eq-thmCasDet}
\E_{B}\cro{\gpi_{\bsX}\pa{\co{\sB}\pa{\theta\et,2^{J}r}}}\le 2\exp(-\xi)\quad \text{with}\quad r=\ray(\gpi,\theta\et)+\frac{A_{0}(\gP\et,B,\theta\et)+1+\xi}{\gamma}
\end{equation}
where $\gamma$ and $J$ are the constants given~\eref{def-gamma} and \eref{def-J} respectively while $A_{0}(\gP\et,B,\theta\et)$ is given by \eref{def-e-r}. This means that a randomised  estimator $\widehat \theta$ with distribution $\gpi_{\bsX}$ satisfies 
\begin{equation}\label{eq-thmCasDet01}
\P_{B}\cro{2^{-J}\ell\pa{\theta\et,\widehat \theta}>\ray(\gpi,\theta\et)+\gamma^{-1}\pa{A_{0}(\gP\et,B,\theta\et)+1+\xi}}\le 2\exp(-\xi).
\end{equation}
In particular, if $\P(B)\ge 1-e^{-\xi}$ then $\widehat \theta$ satisfies
\begin{equation}\label{eq-thmCasDet02}
\P\cro{2^{-J}\ell\pa{\theta\et,\widehat \theta}>\ray(\gpi,\theta\et)+\gamma^{-1}\pa{A_{0}(\gP\et,B,\theta\et)+1+\xi}}\le 3\exp(-\xi).
\end{equation}
\end{thm}
The proof of Theorem~\ref{thmCasDet} is postponed to Section~\ref{sect-PThm1}.

\begin{exa}[The density framework (continued)]
We have seen in our Proposition~\ref{prop-densite} that the test statistic $\gT_{1}$ given by \eref{eq-T1} satisfies our Assumption~\ref{Ass-CasDet} for the choice $(\lambda,\beta)=(0.11, 0.09)$. We may therefore infer from Theorem~\ref{thmCasDet} the following result. For every  $\xi>\log 3$ and $\theta\et\in \Theta$, a randomised estimator $\widehat \theta$ drawn with the distribution $\gpi_{\bsX}$ defined by \eref{def-piX} with $\gT=\gT_{1}$ possesses the property that with  a probability at least $1-3e^{-\xi}$
\begin{equation}\label{thmDPP-00}
Ch^{2}\pa{P_{\theta\et},P_{\widehat{\theta}}}\le h^{2}(P\et,P_{\theta\et})+\frac{r(\gpi,\theta\et)+t_{\xi}+\xi}{n}
\end{equation}
where $C$ is a positive numerical constant. In particular, it follows from the triangle inequality that with a probability at least $1-3e^{-\xi}$
\begin{equation}\label{thmDPP-01}
Ch^{2}\pa{P\et,P_{\widehat \theta}}\le \inf_{\theta\et\in \Theta}\cro{h^{2}(P\et,P_{\theta\et})+\frac{r(\gpi,\theta\et)}{n}}+\frac{t_{\xi}+\xi}{n}.
\end{equation}
%

Some comments are in order: 

It follows from~\eref{thmDPP-00} that $r(\gpi,\theta\et)/n$ is the bound we would get if the data were truly i.i.d.\ with distribution $P_{\theta\et}\in\sM$.

When the data are independent but not i.i.d., hence $X_{i}=X_{i}\et\sim P_{i}\et$ for all $i\in\{1,\ldots,n\}$, we may take $t_{\xi}=0$ and this risk bound inflates by the additional term $h^{2}(P\et,P_{\theta\et})$ where we recall that $P\et=n^{-1}\sum_{i=1}^{n}P_{i}\et$. In particular, $h^{2}(P\et,P_{\theta\et})$ may be 0 even when none of the marginal distributions $P_{i}\et$ belong to $\sM$ provided that their average does. 

When $\etc{X}$ are independent, an integration of \eref{thmDPP-01} with respect to $\xi>0$ yields the risk bound 
\[
C\E\cro{h^{2}\pa{P\et,P_{\widehat \theta}}}\le \inf_{\theta\et\in \Theta}\cro{h^{2}(P\et,P_{\theta\et})+\frac{1+r(\gpi,\theta\et)}{n}}.
\]
The risk of $\widehat \theta$ therefore achieves the best trade-off between the approximation of $P\et$ by an element of the form $P_{\theta\et}\in\sM$ and the $\gpi$-complexity of the model $\Theta$ at $\theta\et$. 

When the dataset $\{X_{1}\et, \ldots, X_{n}\et\}$ is corrupted in the sense that some of the $X_{i}\et$ are replaced by arbitrary (possibly dependent) data, $t_{\xi}$ may be positive. Inequality \eref{thmDPP-01} then shows that the accuracy of our estimator remains of the same order as in the previous situation, provided that the number of such corrupted data points remains sufficiently small compared to $r(\gpi, \theta\et)$, at least with a probability close to 1.
\end{exa}

\section{Estimating the distribution of a Poisson process with covariates}\label{sect-Poisson}
Throughout this section, 
$\cX$ denotes a Borel subset of a complete separable metric space that we equip with the inherited metric and the corresponding Borel $\sigma$-algebra $\cA$.
We shall view a finite point process $X$ on $(\cX,\cA)$ both as a finite random subset of $\cX$ and a finite discrete measure on $(\cX,\cA)$ putting a mass one at each element of $X$. For example, for a measurable subset $A$ of $\cX$, we shall use the notation $X(A)$ when $X$ is viewed as a measure and the notation $|X\cap A|$ when it is viewed as a subset. In particular, $X(\cX)=|X\cap \cX|<+\infty$ a.s.\ since $X$ is finite. Here, we do not exclude the situation where two (or more) points of $X$ coincide, the elements of $X$ should therefore be counted with their multiplicities.

Given a finite measure $\nu$  on $(\cX,\cA)$, we recall that $X$ is a Poisson process with finite intensity measure $\nu$
if for every bounded measurable function $f$ on $(\cX,\cA)$,
\begin{equation}\label{eq-caraPoisson}
\E\cro{\exp\pa{\int_{\cX}fdX}}=\exp\cro{\int_{\cX}\pa{e^{f}-1}d\nu}.
\end{equation}
In particular, by applying this equality to $sf$, with $s\in\R$, and performing a series expansion around 0, we obtain the equality
\[
\E\cro{\int_{\cX}fdX}=\int_{\cX}fd\nu.
\]
A more classical way of defining a Poisson process is by means of this important property. For every disjoint measurable subsets $A_1,\ldots,A_m$ of $\cX$, $X(A_{1}),\ldots,X(A_{m})$ are independent Poisson random variables with means $\nu(A_{1}),\ldots,\nu(A_{m})$ respectively. By convention, a Poisson distribution with mean 0 is the Dirac mass at 0.

\subsection{The statistical framework}
We observe $n\ge 1$ pairs $(w_{1},X_{1}),\ldots,(w_{n},X_{n})$ of random variables where the $w_{i}$ are deterministic covariates with values in $\sW$ and $\etc{X}$ are (almost surely) finite point processes on the measure space $(\cX,\cA,\mu)$. 

In order to model our data, we assume that there exist $n$ independent Poisson point processes $\etc{X\et}$ defined on $(\Omega,\cC)$ with finite intensity measures $\etc{\nu\et}$ respectively, such that the cardinality of the set $X_{i}\ominus X_{i}\et=(X_{i}\setminus X_{i}\et)\cup (X_{i}\et\setminus X_{i})$ is small enough, at least for most of these indices $i\in \{1,\ldots,n\}$. This means that each $X_{i}$ is regarded as a corrupted version of a genuine Poisson point process $X_{i}\et$ from which some points have possibly been removed or added. The misspecification of our model will be partly measured by the random variable
\[
N = \sum_{i=1}^{n} N_{i} \quad \text{with} \quad N_{i} = \ab{X_{i} \ominus X_{i}\et} \quad \text{for } i \in \{1, \ldots, n\},
\]
which counts how many points lie in the symmetric difference between the subsets $X_{i}\et$ and $X_{i}$. The random variable $N$ can also account for the fact that a subset of our dataset may consist of dependent random variables that could be a source of contamination. We denote by $t_{\xi}$ the  $(1-e^{-\xi})$-quantile of $N$ for $\xi>0$. 

To model the intensities $\etc{\nu\et}$, we consider a set $\Theta$ of functions $\theta$ on $\sW\times \cX$ for which  $\theta_{i}:t\mapsto \theta(w_{i},t)$ is nonnegative and integrable on $(\cX,\cA,\mu)$ for every $i\in\{1,\ldots,n\}$. For $\theta\in \Theta$ and $i\in\{1,\ldots,n\}$, $\nu_{\theta,i}=\theta_{i}\cdot\mu=\theta(w_{i},\cdot)\cdot\mu$ thus defines a finite measure on $(\cX,\cA)$. 
We shall proceed as if there existed $\theta\et\in \Theta$ such that $\nu_{i}\et=\nu_{\theta\et,i}$ for all $i\in\{1,\ldots,n\}$.
Our goal is therefore to estimate $\theta\et$ from the observation $\bsX=(X_{1},\ldots,X_{n})$ and the knowledge of $w_{1},\ldots,w_{n}$. 

Given two finite measures $\nu$ and $\nu'$ which are dominated by $\overline \nu$ on $(\cX,\cA)$, we set 
\begin{equation}
  \label{eq-def-H}
  H^{2}(\nu,\nu')=\frac{1}{2}\int_{\cX}\pa{\sqrt{\frac{d\nu}{d\overline \nu}}-\sqrt{\frac{d\nu'}{d\overline \nu}}}^{2}d\overline \nu,
\end{equation}
the result being independent of the choice of $\overline \nu$. 
We write $\gnu\et = (\etc{\nu\et})$ and for an $n$-tuple $\gnu=(\etc{\nu})$ of finite measures on $(\cX,\cA)$ and $\theta\in \Theta$,
\[
\gnu_{\theta}=(\nu_{\theta,1},\ldots,\nu_{\theta,n})\quad \text{and}\quad \gH^{2}(\gnu,\gnu_{\theta})=\sum_{i=1}^{n}H^{2}(\nu_{i},\nu_{\theta,i}).
\]
We equip our parameter space $\Theta$ with the loss $\ell(\theta,\theta')=\gH^{2}(\gnu_{\theta},\gnu_{\theta'})$ for $\theta,\theta'\in \Theta$ and a $\sigma$-algebra that makes the corresponding balls measurable. This turns $\Theta$ into a measurable space that can be equipped with a prior $\gpi$. 

\subsection{The test statistic and the property of the posterior}
We use the test statistic defined for $\theta,\theta'\in \Theta$ by 
\begin{equation}\label{eq-TPP}
\gT_{2}(\bsX,\theta,\theta')=\sum_{i=1}^{n}\cro{\int_{\cX}\psi\pa{\sqrt{\frac{\theta_{i}'}{\theta_{i}}}}dX_{i}+\frac{1}{4}\pa{\int_{\cX}\theta_{i} d\mu-\int_{\cX}\theta_{i}'d\mu}}.
\end{equation}
A different one was proposed in Sart~\citeyearpar{MR3394490} in the frequentist setting to define $T$-estimators. 
Given a realization $\gx = (x_1,\ldots,x_n)$ where each $x_i$ is a subset of $\mathcal X$ (also viewed as a finite discrete measure),
we note that 
$|\sum_{i=1}^{n}\int_{\cX}\psi(\sqrt{\theta_{i}'/\theta_{i}})dx_{i}|
\le \sum_{i=1}^{n} |x_i|
$, hence Assumption~\ref{Ass-posterior-well-defined} is satisfied.
The following result shows Assumption~\ref{Ass-CasDet} is also met. 

\begin{prop}\label{prop-Poisson}
Let $\xi>0$ and $\theta\et$ be an arbitrary point in $\Theta$. Assumption~\ref{Ass-CasDet} is satisfied at $\theta\et$ with $\gT=\gT_{2}$, $B= \{N\le t_{\xi}\}$, 
$\ell(\theta,\theta')=\gH^{2}(\gnu_{\theta},\gnu_{\theta'})$ for $\theta,\theta'\in \Theta$ provided that  $(\lambda,\beta)\in (0,+\infty)\times (0,1)$ satisfies the constraints given by \eref{eq-constraints1}.
We recall that the values $\lambda=0.11$ and $\beta=0.09$ suit, in which case one may take $\gamma=7.9\times 10^{-4}$ and $J=13$. The expressions of the numerical constants $\frc_{i}(\lambda,\kappa)$ for $i\in\{0,1,2,3\}$ and $\kappa\in\{\beta,\overline \beta\}$ are provided in the proof of Proposition~\ref{prop-densite} in Section~\ref{prf-prop-densite} and $A_{0}(\lambda,\kappa,\gP\et,B,\theta\et)=\lambda (\kappa+1)t_{\xi}-\log(1-e^{-\xi})+\frc_{0}(\lambda,\kappa)\gH^{2}(\gnu\et,\gnu_{\theta\et})$. 
\end{prop}

Applying Proposition~\ref{prop-Poisson}, we derive from Theorem~\ref{thmCasDet} the following result. 
\begin{thm}\label{thm-Poisson}
Let $\xi>\ln 3$, $\theta\et\in \Theta$ and $(\lambda,\beta)\in (0,+\infty)\times (0,1)$ that satisfies the constraints given by~\eref{eq-constraints1}. An estimator $\widehat \theta$ drawn with the posterior distribution $\gpi_{\bsX}$ defined by \eref{def-piX} with $\gT=\gT_{2}$ given in \eref{eq-TPP}, possesses the following property. With a probability at least $1-3e^{-\xi}$
\begin{equation}\label{thmP-00}
C\gH^{2}\pa{\gnu_{\theta\et},\gnu_{\widehat \theta}}\le \gH^{2}\pa{\gnu\et,\gnu_{\theta\et}}+r(\gpi,\theta\et)+t_{\xi}+\xi
\end{equation}
where $C=C(\lambda,\beta)>0$ is a numerical constant that only depends on the chosen values of $\lambda$ and $\beta$. In particular, with a probability at least $1-3e^{-\xi}$
\begin{equation}\label{thmP-01}
C'\gH^{2}\pa{\gnu\et,\gnu_{\widehat \theta}}\le \inf_{\theta\et\in \Theta}\cro{\gH^{2}\pa{\gnu\et,\gnu_{\theta\et}}+r(\gpi,\theta\et)}+t_{\xi}+\xi
\end{equation}
for some $C'=C'(\lambda,\beta)>0$.
\end{thm}

\subsection{An example}\label{sect-ex}
In this example $\sW$ denotes the unit sphere of $\R^{k}$, where the dimension $k\ge 2$ is typically large, and $\cX$ a Borel subset of $\R^{d}$, $d\ge 1$, equipped with the Lebesgue measure $\mu$. We observe $n$ finite point processes $\etc{X}$ on $\cX$ and assume that $X_{i}$ is close to a Poisson process $X_{i}\et$ on $\cX$ in the sense  $N_{i}=|X_{i}\ominus X_{i}\et|$ is small enough, at least for most of the indices $i$. The $X_{i}\et$ are assumed to be independent and when $X_{i}$ is associated with a covariate $w_{i}\in\sW$, we assume that the intensity $\nu_{i}\et$ of $X_{i}\et$ admits a density $\theta_{i}\et=\theta\et(w_{i},\cdot)$ with respect to $\mu$ which is of the form $\theta_{i}\et:x\mapsto \pa{\rho\et\scal{w_{i}}{a\et}s\et(x)}^{2}$ where  $\rho\et$ is a positive number, $s\et$ belongs to the unit sphere $\S$ of the Hilbert space $\H=\L_{2}(\cX,\cA,\mu)$ whilst $a\et$ belongs to $\sW$. Note that the average cardinality of $X_{i}\et$ is $\int_{\cX}\pa{\rho\et\scal{w_{i}}{a\et}s\et(x)}^{2}d\mu(x)=\pa{\rho\et\scal{w_{i}}{a\et}}^{2}$. It therefore ranges in $[0,(\rho\et)^{2}]$ and depends on the angle between the covariate $w_{i}$ and the direction $a\et$. The Hilbert norm of $\H$ will be denoted $\norm{\cdot}$ throughout this section.

Since $k$ is assumed to be large, we wish to reduce the dimensionality of the problem by assuming some sparsity on the parameter $a\et$. This means that the parameter $a\et$ belongs to a set of the form $A_{J}=\{a\in\sW|\; a_{j}=0, \text{ for } j\not\in J\}$ for a certain nonempty subset $J=J\et$ of $\{1,\ldots,k\}$ the cardinality of which is expected to be small compared to $k$.  We denote by $\cJ$ the class of these nonempty subsets of $\{1,\ldots,k\}$. 

To approximate $s\et$, we consider an at most countable collection $\{V_{m},\; m\in\cM\}$ of finite dimensional linear subspaces of $\H$ that we index with a set denoted by $\cM$. We assume that for every $m\in\cM$, the dimension $D_{m}$ of $V_{m}$ is at least one so that its unit sphere $S_{m}$ is non-empty. 

Our parameter space $\Theta$ therefore consists of these functions $\theta$ on $\sW\times\cX$ which are of the form $\theta_{\rho,a,s}:(w,x)\mapsto \pa{\rho\scal{a}{w}s(x)}^{2}$ where the parameter $(\rho,a,s)$ belongs to one of the sets $\Gamma_{J,m}=(0,+\infty)\times A_{J}\times S_{m}$ with $(J,m)\in\cJ\times \cM$, that is,
\[
\Theta=\ac{\theta_{\rho,a,s},\; (\rho, a,s)\in \bigcup_{(J,m)\in\cJ\times\cM}\Gamma_{J,m}}.
\]
To keep our notation as simple as possible, we denote by $\gnu_{\rho,a,s}$ the $n$-tuple of intensities $(\nu_{\theta,1},\ldots,\nu_{\theta,n})$  when $\theta=\theta_{\rho,a,s}$. For $(\rho,a,s)$ and $(\varrho,b,t)$ in $(0,+\infty)\times \sW\times \S$, we note that 
\begin{align*}
2\gH^{2}\pa{\gnu_{\rho,a,s},\gnu_{\varrho,a,s}}&=\pa{\rho-\varrho}^{2}\sum_{i=1}^{n}\int_{\cX}\pa{\sqrt{\pa{\scal{a}{w_{i}}s(x)}^{2}}}^{2}d\mu(x)\\
&=\pa{\rho-\varrho}^{2}\sum_{i=1}^{n}\ab{\scal{a}{w_{i}}}^{2}\int_{\cX}s^{2}(x)d\mu(x)\le n\pa{\rho-\varrho}^{2}.
\end{align*}
We obtain similarly that  

\begin{align}
2\gH^{2}\pa{\gnu_{\rho,a,s},\gnu_{\rho,a,t}}&=\rho^{2}\sum_{i=1}^{n}\int_{\cX}\pa{\sqrt{\pa{\scal{a}{w_{i}}s(x)}^{2}}-\sqrt{\pa{\scal{a}{w_{i}}t(x)}^{2}}}^{2}d\mu(x)\nonumber\\
&=\rho^{2}\sum_{i=1}^{n}\ab{\scal{a}{w_{i}}}^{2}\int_{\cX}\pa{\ab{s(x)}-\ab{t(x)}}^{2}d\mu\nonumber\\
&\le \rho^{2}\norm{s-t}^{2}\sum_{i=1}^{n}\ab{\scal{a}{w_{i}}}^{2}\le n \rho^{2}\norm{s-t}^{2},\label{H-st}
\end{align}
whilst, since $\norm{s}=1$, 
\begin{align*}
2\gH^{2}\pa{\gnu_{\rho,a,s},\gnu_{\rho,b,s}}&=\rho^{2}\sum_{i=1}^{n}\int_{\cX}\pa{\sqrt{\pa{\scal{a}{w_{i}}s(x)}^{2}}-\sqrt{\pa{\scal{b}{w_{i}}s(x)}^{2}}}^{2}d\mu(x)\\
&=\rho^{2}\sum_{i=1}^{n}\pa{\ab{\scal{a}{w_{i}}}-\ab{\scal{b}{w_{i}}}}^{2}\le  \rho^{2}\sum_{i=1}^{n}\ab{\scal{a-b}{w_{i}}}^{2}\le n\rho^{2}\ab{a-b}^{2}.
\end{align*}
Consequently, for every $(\rho,a,s)$ and $(\varrho,b,t)$ in $(0,+\infty)\times \sW\times \S$
\begin{align}
\gH^{2}\pa{\gnu_{\rho,a,s},\gnu_{\varrho,b,t}}&\le \pa{\gH\pa{\gnu_{\rho,a,s},\gnu_{\rho,b,s}}+\gH\pa{\gnu_{\rho,b,s},\gnu_{\rho,b,t}}+\gH\pa{\gnu_{\rho,b,t},\gnu_{\varrho,b,t}}}^{2}\nonumber\\
&\le 3\cro{\gH^{2}\pa{\gnu_{\rho,a,s},\gnu_{\rho,b,s}}+\gH^{2}\pa{\gnu_{\rho,b,s},\gnu_{\rho,b,t}}+\gH^{2}\pa{\gnu_{\rho,b,t},\gnu_{\varrho,b,t}}}\nonumber\\
&\le \frac 32 n\cro{\rho^{2}\ab{a-b}^{2}+\rho^{2}\norm{s-t}^{2}+\ab{\rho-\varrho}^{2}}.\label{eq-H-E}
\end{align}

Let us now define a prior on $\Theta$. 
Given $(J,m)\in\cJ\times \cM$, we consider an orthonormal basis $\phi_{1,m},\ldots,\phi_{D_{m},m}$ of $V_{m}$ and three independent random variables $R$, $\bsY_{\!\!J}=(Y_{1,J},\ldots,Y_{k,J})$ and $\bsZ_{m}=(Z_{1,m},\ldots Z_{D_{m},m})$ with values in $(0,+\infty)$, $\sW$ and the unit sphere of $\R^{D_{m}}$ respectively. The random variable $R$ has density $x\mapsto (2/\pi)(1+x^{2})^{-1}$ on $(0,+\infty)$,  $\bsZ_{m}$ and $(Y_{j,J})_{j\in J}$ are uniformly distributed on the unit spheres of $\R^{D_{m}}$ and $\R^{|J|}$ respectively whilst $Y_{j,J}=0$ for $j\not\in J$. We denote by $\gpi_{J,m}$ the image of the distribution of the random variable $(R,Y_{\!\!J},\sum_{j=1}^{D_{m}}Z_{j,m}\phi_{j,m})\in \Gamma_{J,m}$ by the mapping $(\rho,a,s)\mapsto \theta_{\rho,a,s}$ which yields a probability measure on $\Theta$. Given a collection $\{L_{J,m},\; (J,m)\in\cJ\times \cM\}$ of nonnegative weights that satisfy $\sum_{(J,m)\in\cJ\times \cM}e^{-L_{J,m}}=1$,  we define our prior on $\Theta$ as 
\begin{equation}\label{eq-piex}
\gpi=\sum_{(J,m)\in\cJ\times \cM}e^{-L_{J,m}}\gpi_{J,m}.
\end{equation}

The following proposition holds. 
\begin{prop}\label{prop-piex}
Let $(J,m)\in \cJ\times\cM$ and $(\rho,a,s)\in \Gamma_{J,m}=(0,+\infty)\times A_{J}\times S_{m}$. Then, 
\begin{align}
\overline r\pa{\gpi,\theta_{\rho,a,s}}&\le \frac{|J|+D_{m}}{\gamma}\log\pa{e+3(1+\rho)\sqrt{\frac{ \gamma n}{(|J|+D_{m})\vee L_{J,m}}}}\label{eq-prop-piex}\\
&\hspace{1cm }+\frac{1}{\gamma}\log\pa{\frac{\pi^{2}(1+\rho)\sqrt{|J|D_{m}}}{2}}+\frac{L_{J,m}}{\gamma}\nonumber.
\end{align}
Besides, if $s\et$ is an arbitrary point on the unit sphere of $\H$,
\begin{equation}\label{eq-prop-approx}
\inf_{s\in S_{m}}\gH^{2}\pa{\gnu_{\rho,a,s\et},\gnu_{\rho,a,s}}\le n\rho^{2}\inf_{v\in V_{m}}\norm{s\et-v}^{2}=n\inf_{v\in V_{m}}\norm{\rho s\et-v}^{2}.
\end{equation}
\end{prop}
We deduce from Theorem~\ref{thm-Poisson} the following result. 
\begin{thm}\label{thm-Poissonex}
Let $\xi>\ln 3$, $(\lambda,\beta)\in (0,+\infty)\times (0,1)$ that satisfies the constraints given by~\eref{eq-constraints1}. An estimator $\widehat \theta$ drawn with the posterior distribution $\gpi_{\bsX}$ defined by \eref{def-piX} with $\gT=\gT_{2}$ given in \eref{eq-TPP}, possesses the following property. With a probability at least $1-3e^{-\xi}$
\begin{align*}
&Cn^{-1}\gH^{2}\pa{\gnu_{\theta_{\rho\et,a\et,s\et}},\gnu_{\widehat \theta}}\\
&\le \inf_{m\in\cM}\cro{\inf_{v\in V_{m}}\norm{\rho\et s\et-v}^{2}+\frac{|J\et|+D_{m}}{n}\log_{+}\cro{\frac{n(1+\rho\et)^{2}}{(|J\et|+D_{m})\vee L_{J\et,m}}}+\frac{L_{J\et,m}}{n}}+\frac{t_{\xi}+\xi}{n},
\end{align*}
where $C=C(\lambda,\beta)>0$ is a numerical constant that only depends on the chosen values of $\lambda$ and $\beta$. \end{thm}

For $m\in \cM$, the quantity $\inf_{v\in V_{m}}\norm{\rho\et s\et-v}^{2}$ corresponds to the approximation  of $\rho\et s\et$ by an element of our linear space $V_{m}$. Up to the logarithmic term, $(|J\et|+D_{m})/n$ is the statistical error we would get for estimating the parameter $(\rho\et,a\et,s\et)$ if it belonged to $\Gamma_{J\et,m}$. The quantity $t_{\xi}$ accounts for the robustness of the procedure with respect to a possible departure from the assumption that our data $\etc{X}$ are truly independent Poisson point processes. Finally, the term $L_{J\et,m}$ depends on our choice of the weights. For example, denoting by $K_{D}$ the quantity $\ab{\{m\in\cM,\; D_{m}=D\}}$ for $D\ge 1$,  one may take for every $(J,m)\in\cJ\times \cM$, $L_{J,m}=L_{1,|J|}+L_{2,D_{m}}$ with, for $j\in\{1,\ldots,k\}$ and $D\ge 1$,  
\[
L_{1,j}=\log\binom{k}{j}+\log k\; \text{ and }\; L_{2,D}=\log\pa{K_{D}\vee 1}+D+\log\pa{\sum_{D'\ge 1}(K_{D'}\wedge 1)e^{-D'}}.
\]
Since $K_{D}/(K_{D}\vee 1)=K_{D}\wedge 1$ for every $D\ge 1$, we can check that
\begin{align*}
\sum_{(J,m)\in\cJ\times \cM}e^{-L_{J,m}}&=\sum_{j=1}^{k}\binom{k}{j}e^{-L_{1,j}}\sum_{D\ge 1}K_{D}e^{-L_{2,D}}\\
&=\sum_{j=1}^{k}\frac{\binom{k}{j}}{\binom{k}{j}k}\sum_{D\ge 1}\frac{K_{D}}{K_{D}\vee 1}\frac{e^{-D}}{\sum_{D'\ge 1}(K_{D'}\wedge 1)e^{-D'}}=1.\\
\end{align*}
If there exists $c>0$ such that  $K_{D}\le \exp(cD)$ for every $D\ge 1$, it follows from the inequality $\binom{k}{j}\le k^{j}$ for $j\in\{1,\ldots,k\}$ that $L_{J,m}\le (|J|+1)\log k+(c+1)D_{m}$ for every $(J,m)\in\cJ\times \cM$. For such a choice, our previous bound $n^{-1}\gH^{2}\pa{\gnu_{\theta_{\rho\et,a\et,s\et}},\gnu_{\widehat \theta}}$ remains then of the same order up to a possible additional term of order $(|J\et|+1)\log k/n$.

\subsection{The case of a Poisson process with no covariates}
Let us now model our observations  $\etc{X}$ by making the assumption that they are simply $n$ i.i.d.\ Poisson processes and that their common intensity $\nu_{\theta\et}$ is of the form $\theta\et\cdot\mu$ where $\theta\et$ belongs to a given class $\Theta$ of nonnegative integrable functions on $(\cX,\cA,\mu)$. If $\etc{X}$ were genuine i.i.d.\ Poisson processes with the same intensity $\nu_{\theta\et}$, their sum $S=\sum_{i=1}^{n}X_{i}$ would also be a Poisson process but with intensity $n\nu_{\theta\et}=(n\theta\et)\cdot\mu$. In place of the $n$-tuple $(\etc{X})$ we may alternatively consider the sole observation $S$ and apply our procedure to $\bsX=S$ and the family of intensities  $\{(n\theta)\cdot \mu,\; \theta\in \Theta\}$. In fact, taking either $S$ or $(\etc{X})$ as our observed data has actually no impact on the resulting estimator since the test statistic $\gT_{2}$ given by \eref{eq-TPP} can equivalently be written as
\[
\gT_{2}(\bsX,\theta,\theta')=\gT_{2}(S,\theta,\theta')=\int_{\cX}\psi\pa{\sqrt{\frac{n\theta'}{n\theta}}}dS+\frac{1}{4}\pa{\int_{\cX}(n\theta) d\mu-\int_{\cX}(n\theta')d\mu}.
\]
Moreover, for two intensity measures $\nu,\nu'$
\[
H^{2}(n\nu,n\nu')=nH^{2}(\nu,\nu')=\gH^{2}(\gnu,\gnu')
\]
 with $\gnu=(\nu,\ldots,\nu)$ and $\gnu'=(\nu',\ldots,\nu')$. Applying our Theorem~\ref{thm-Poisson} either to the single observation $S$ or the $n$-tuple $(\etc{X})$ yields 
\begin{cor}\label{cor-Poisson}
Let $\xi>\ln 3$, $\theta\et\in \Theta$ and $(\lambda,\beta)\in (0,+\infty)\times (0,1)$ that satisfies the constraints given by~\eref{eq-constraints1}. An estimator $\widehat \theta$ drawn with the posterior distribution $\gpi_{\bsX}$ defined by \eref{def-piX} with $\gT=\gT_{2}$, possesses the following property. With a probability at least $1-3e^{-\xi}$
\begin{equation}\label{thmP-00b}
CH^{2}\pa{\nu_{\theta\et},\nu_{\widehat \theta}}\le H^{2}\pa{\nu\et,\nu_{\theta\et}}+\frac{r(\gpi,\theta\et)+t_{\xi}+\xi}{n}
\end{equation}
where $C=C(\lambda,\beta)>0$ is a numerical constant that only depends on the chosen values of $\lambda$ and $\beta$. In particular, it follows from the triangle inequality that with a probability at least $1-3e^{-\xi}$
\begin{equation}\label{thmP-01b}
C'H^{2}\pa{\nu\et,\nu_{\widehat \theta}}\le \inf_{\theta\et\in \Theta}\cro{H^{2}\pa{\nu\et,\nu_{\theta\et}}+\frac{r(\gpi,\theta\et)}{n}}+\frac{t_{\xi}+\xi}{n}
\end{equation}
for some $C'=C'(\lambda,\beta)>0$.
\end{cor}

\section{Proof of Theorem~\ref{thmCasDet}}\label{sect-PThm1}
Let $\theta\et\in \Theta$, $B$ be an event with positive probability and $\sA$ a measurable subset of  $\Theta$. We use the notation $\gtheta$ for a triplet $(\theta_{1},\theta_{2},\theta_{3})$ in $\Theta^{3}$. We shall often use the following lemma which is proven in Audibert and Catoni~\citeyearpar{audibert2011linear} [Lemma 4.2, page 28].
\begin{lem}\label{lem-Catoni}
Let $(U,V)$ be a pair of random variables with values in a product space $(E\times F,\cE\otimes \cF)$ and marginal distributions $P_{U}$ and $P_{V}$ respectively. For all measurable function $h$ on $(E\times F,\cE\otimes \cF)$, 
\[
\E_{U}\cro{\frac{1}{\E_{V}\cro{\exp\cro{-h(U,V)}}}}\le \cro{\E_{V}\cro{\frac{1}{\E_{U}\cro{\exp\cro{h(U,V)}}}}}^{-1}.
\]
\end{lem}

Given a positive number $z$ to be chosen later on, we consider the event 
\[
A=\ac{\int_{\Theta}\exp\cro{-\lambda(1-\beta) \gT(\bsX,\theta)}d\gpi(\theta)>z}\subset \Omega.
\]
By Markov's inequality, the probability of its complement satisfies  
\begin{align}
\P_{B}\pa{\co{A}}&=\P_{B}\cro{\pa{\int_{\Theta}\exp\cro{-\lambda(1-\beta) \gT(\bsX,\theta)}d\gpi(\theta)}^{-1}\ge z^{-1}}\nonumber\\
&\le z\E_{B}\cro{\pa{\int_{\Theta}\exp\cro{-\lambda(1-\beta) \gT(\bsX,\theta)}d\gpi(\theta)}^{-1}}.\label{eq-00}
\end{align}

We derive from the definition~\eref{def-piX} of $\gpi_{\bsX}$ that 
\begin{align*}
\E_{B}\cro{\gpi_{\bsX}\pa{\sA}}&=\E_{B}\cro{\gpi_{\bsX}\pa{\sA}\1_{\co{A}}(\bsX)}+\E_{B}\cro{\gpi_{\bsX}\pa{\sA}\1_{A}(\bsX)}\\
&\le \P_{B}(\co{A})+z^{-1}\E_{B}\pa{\int_{\sA}\exp\cro{-\lambda(1-\beta) \gT(\bsX,\theta)}d\gpi(\theta)}\\
&= \P_{B}(\co{A})+z^{-1}\int_{\sA}\E_{B}\pa{\exp\cro{-\lambda(1-\beta) \gT(\bsX,\theta)}}d\gpi(\theta).
\end{align*}
Using  \eref{eq-00}, we deduce that $\E_{B}\cro{\gpi_{\bsX}\pa{\sA}}\le z \cE_{1}+z^{-1}\cE_{2}(\sA)$
with 
\begin{align}
\cE_{1}&=\E_{B}\cro{\pa{\int_{\Theta}\exp\cro{-\lambda(1-\beta) \gT(\bsX,\theta)}d\gpi(\theta)}^{-1}}\label{def-cE1}\\
\cE_{2}(\sA)&=\E_{B}\cro{\int_{\sA} \exp\pa{-\lambda(1-\beta) \gT(\bsX,\theta)}d\gpi(\theta)}.\label{def-cE2}
\end{align}
If $\cE_{2}(\sA)>0$ we pick $z=\sqrt{\cE_{2}(\sA)/\cE_{1}}$ and we get
\begin{equation}\label{eq-020}
\E_{B}\cro{\gpi_{\bsX}\pa{\sA}}\le 2\sqrt{\cE_{1}\cE_{2}(\sA)}.
\end{equation}
If $\cE_{2}(\sA)=0$, the conclusion follows by letting $z \downarrow 0$.

We bound $\cE_{1}$ and $\cE_{2}(\sA)$ by means of the following proposition.
\begin{prop}\label{prop-f}
For every $\lambda>0$ and $\beta\in (0,1)$,  
\begin{equation}\label{eq-prop-f1}
\cE_{1}\le \cro{\int_{\Theta}\cro{\int_{\Theta}\pa{\int_{\Theta}\frac{d\gpi(\theta_{1})}{\sL_{\overline \beta}(\gtheta,\lambda|B)}}^{-1}d\gpi(\theta_{2})}^{-1}d\gpi(\theta_{3})}^{-1}
\end{equation}
and for every measurable subset $\sA$ of $\Theta$,  
\begin{align}
\cE_{2}(\sA)\le \int_{\sA}\cro{\int_{\Theta}\pa{\int_{\Theta}\frac{d\gpi(\theta_{1})}{\sL_{\beta}(\gtheta,\lambda| B)}}^{-1}d\gpi(\theta_{2})}d\gpi(\theta_{3}).\label{eq-prop-f4}
\end{align}
\end{prop}

\begin{proof}
It follows from \eref{def-TX} and the convexity of the exponential that for every $s\in \R$, 
\begin{align*}
\E_{B}\cro{\exp\pa{\lambda s\gT(\bsX,\theta)}}&=\E_{B}\cro{\exp\pa{\int_{\Theta}\lambda s\gT(\bsX,\theta,\theta')\frac{\exp\cro{\lambda\gT(\bsX,\theta,\theta')}d\gpi(\theta')}{\int_{\Theta}\exp\cro{ \lambda \gT(\bsX,\theta,\theta'')}d\gpi(\theta'')}}}\\
&\le \int_{\Theta}\E_{B}\cro{\frac{\exp\cro{\lambda(s+1)\gT(\bsX,\theta,\theta')}}{\int_{\Theta}\exp\cro{ \lambda \gT(\bsX,\theta,\theta'')}d\gpi(\theta'')}}d\gpi(\theta').
\end{align*}
This yields
\begin{align*}
&\E_{B}\cro{\exp\pa{\lambda s\gT(\bsX,\theta)}}\\
&\le \int_{\Theta}\E_{B}\cro{\pa{\int_{\Theta}\exp\cro{\lambda\pa{\gT(\bsX,\theta,\theta'')-(s+1)\gT(\bsX,\theta,\theta')}}d\gpi(\theta'')}^{-1}}d\gpi(\theta')\\
&=\int_{\Theta}\E_{B}\cro{\pa{\int_{\Theta}\exp\cro{-\lambda\Delta_{s+1}(\bsX,\gtheta)}d\gpi(\theta_{1})}^{-1}}d\gpi(\theta_{2}).
\end{align*}
By applying Lemma~\ref{lem-Catoni}, we deduce that
\begin{align}
\E_{B}\cro{\exp\pa{\lambda s\gT(\bsX,\theta)}}&\le \int_{\Theta}\cro{\int_{\Theta}{\frac{d\gpi(\theta_{1})}{\E_{B}\cro{\exp\cro{\lambda\Delta_{s+1}(\bsX,\gtheta)}}}}}^{-1} d\gpi(\theta_{2})\nonumber\\
&=\int_{\Theta}\cro{\int_{\Theta}\frac{d\gpi(\theta_{1})}{\sL_{s+1}(\gtheta,\lambda|B)}}^{-1}d\gpi(\theta_{2}) \label{eq-prop-f1b}.
\end{align}

Let us now bound $\cE_{1}$. By applying  Lemma~\ref{lem-Catoni} again and \eref{eq-prop-f1b} with $s=1-\beta$, we obtain that
\begin{align*}
\cE_{1}&=\E_{B}\cro{\pa{\int_{\Theta}\exp\cro{-\lambda s \gT(\bsX,\theta)}d\gpi(\theta)}^{-1}}\le \cro{\int_{\Theta}\frac{d\gpi(\theta)}{\E_{B}\cro{\exp\pa{\lambda s \gT(\bsX,\theta)}}}}^{-1}\\
&\le \cro{\int_{\Theta}\cro{\int_{\Theta}\pa{\int_{\Theta}\frac{d\gpi(\theta_{1})}{\sL_{2-\beta}(\gtheta,\lambda|B)}}^{-1}d\gpi(\theta_{2})}^{-1}d\gpi(\theta_{3})}^{-1}
\end{align*}
which yields \eref{eq-prop-f1}.

Let us now turn to $\cE_{2}(\sA)$. Let us first observe that 
\begin{align*}
\cE_{2}(\sA)&=\E_{B}\cro{\int_{\sA}\exp\pa{-\lambda(1-\beta) \gT(\bsX,\theta)}d\gpi(\theta)}\\
&=\int_{\sA}\E_{B}\cro{\exp\pa{-\lambda(1-\beta) \gT(\bsX,\theta)}}d\gpi(\theta)
\end{align*}
and the result follows by applying \eref{eq-prop-f1b} with $s=-(1-\beta)$.
\end{proof}

\subsection{End of the proof of Theorem~\ref{thmCasDet}}
Throughout this section, we drop the dependency with respect to $\lambda$ of  the coefficients $A_{0}(\lambda,\kappa,\gP\et,B,\theta\et), \frc_{i}(\lambda,\kappa)$ for $i\in\{1,2,3\}$. We also drop the dependency of $A_{0}$ with respect to $\gP\et$ and $\theta\et$ so that  we write for short $A_{0}(\kappa), \frc_{i}(\kappa)$ for $i\in\{1,2,3\}$. We control the quantities $\cE_{1}$ and $\cE_{2}(\sA)$ defined by~\eref{def-cE1} and~\eref{def-cE2} respectively by means of the following proposition.

\begin{prop}\label{prop-det}
Let $\lambda>0$, $\beta\in (0,1)$, $\overline \beta=2-\beta$ and $B$ an event with positive probability. Assume that Assumption~\ref{Ass-CasDet} is satisfied at $\theta\et$ and set for $\kappa\in \{\beta,\overline \beta\}$, $I_{0,\kappa}=\exp\pa{A_{0}(\kappa)}$ and for a given measurable subset $\sC$ of $\Theta$ 
\begin{align}\label{def-Ii}
I_{i,\kappa}(\sC)&=\int_{\sC}\exp\cro{-|\frc_{i}(\kappa)|\ell\pa{\theta\et,\theta}}d\gpi(\theta)\quad \text{for $i\in\{1,2,3\}$.}
\end{align}
Then, for every measurable subset $\sA$ of $\Theta$,
\begin{align}\label{eq-cE}
\cE_{1}&\le I_{0,\overline \beta}I_{1,\overline \beta}^{-1}(\Theta)I_{2,\overline \beta}(\Theta)I_{3,\overline \beta}^{-1}(\Theta)\quad \text{and}\quad \cE_{2}(\sA)\le I_{0,\beta}I_{1,\beta}^{-1}(\Theta)I_{2,\beta}(\Theta)I_{3,\beta}(\sA).
\end{align}

\end{prop}

\begin{proof}
Under Assumption~\ref{Ass-CasDet}, 
\begin{align*}
\sL_{\kappa}(\gtheta,\lambda|B)\le  \exp\pa{A_{0}(\kappa)}\prod_{i=1}^{3}\exp\cro{(-1)^{i-1}\frc_{i}(\kappa)\ell\pa{\theta\et,\theta_{i}}}.
\end{align*}
Applying inequality~\eref{eq-prop-f1} of Proposition~\ref{prop-f} with Fubini's theorem, we get 
\begin{align*}
\cE_{1}\le \cro{\int_{\Theta}\cro{\int_{\Theta}\pa{\int_{\Theta}\frac{d\gpi(\theta_{1})}{\sL_{\overline \beta}(\gtheta,\lambda|B)}}^{-1}d\gpi(\theta_{2})}^{-1}d\gpi(\theta_{3})}^{-1}
&\le I_{0,\overline \beta}I_{1,\overline\beta}^{-1}(\Theta)I_{2,\overline\beta}(\Theta)I_{3,\overline\beta}^{-1}(\Theta).
\end{align*}
By arguing similarly, we derive from \eref{eq-prop-f4} that 
\begin{align*}
\cE_{2}(\sA)
\le \int_{\sA}\cro{\int_{\Theta}\pa{\int_{\Theta}\frac{d\gpi(\theta_{1})}{\sL_{\beta}(\gtheta,\lambda|B)}}^{-1}d\gpi(\theta_{2})}d\gpi(\theta_{3})
\le I_{0,\beta}I_{1,\beta}^{-1}(\Theta)I_{2,\beta}(\Theta)I_{3,\beta}(\sA),
\end{align*}
where the last inequality holds since $\frc_{3}(\beta)<0$. 
\end{proof}

The following lemma allows us to control the integrals defined by\eref{def-Ii}.
\begin{lem}\label{lem-som}
Let $J\in\N$, $\eta>0$ and $\theta\in \Theta$. For every $r\ge \max\{\ray(\gpi,\theta),\gamma^{-1}\}$ 
\begin{align}
\int_{\co{\sB}(\theta,2^{J}r)}\exp\pa{-(2+\eta)\gamma\ell(\theta,\theta')}d\gpi(\theta')\le  \exp\cro{\Xi(\eta)}\gpi\pa{\sB(\theta,r)}\exp\pa{-2^{J}\eta\gamma r}\label{eq-lem-som01}
\end{align}
with $\Xi(\eta)=  -\log\cro{1-\exp\pa{-\eta}}$. Besides, 
\begin{align}
\int_{\Theta}\exp\pa{-(2+\eta)\gamma\ell(\theta,\theta')}d\gpi(\theta')\le \exp\cro{\Xi(\eta)}\gpi\pa{\sB(\theta,r)}.\label{eq-lem-som02}
\end{align}
\end{lem}

\begin{proof}
Let us set $\sB=\sB(\theta,r)$. Since  $r\ge \ray(\gpi,\theta)$, we deduce from~\eref{prop-epsn} that by induction, for every $j\ge 0$ 
\begin{align*}
\gpi\pa{\sB(\theta,2^{j+1}r)}&\le \exp\pa{\gamma\ray\sum_{k=0}^{j}2^{k}}\gpi\pa{\sB}= \exp\pa{(2^{j+1}-1)\gamma\ray}\gpi\pa{\sB}.
\end{align*}
Consequently, 
\begin{align*}
&\int_{\co{\sB}(\theta,2^{J}r)}\exp\cro{-(2+\eta)\gamma\ell(\theta,\theta')}d\gpi(\theta')\\
&=\sum_{j\ge J}\int_{\sB(\theta,2^{j+1}r)\setminus \sbB(\theta,2^{j}r)}\exp\cro{-(2+\eta)\gamma\ell(\theta,\theta')}d\gpi(\theta')\\
&\le \gpi\pa{\sB}\sum_{j\ge J}\frac{\gpi\pa{\sB(\theta,2^{j+1}r)}}{\gpi\pa{\sB}}\exp\cro{-(2+\eta)\gamma 2^{j}r}\\
&\le  \gpi\pa{\sB}\sum_{j\ge J}\exp\cro{ (2^{j+1}-1)\gamma r-(2+\eta)\gamma2^{j}r}\\
&=  \gpi\pa{\sB}\exp\pa{-\gamma r} \sum_{j\ge J}\exp\pa{-\eta\gamma 2^{j}r}.
\end{align*}
This yields 
\[
\int_{\co{\sB}(\theta,2^{J}r)}\exp\cro{-(2+\eta)\gamma\ell(\theta,\theta')}d\gpi(\theta')\le \gpi\pa{\sB}\sum_{j\ge 0}\exp\pa{-\eta\gamma 2^{J+j}r}.
\]
Using the inequality $2^{j}\ge j+1$ for all $j\ge 0$ and the fact that $\gamma r\ge  1$, we obtain that 
\begin{align*}
\sum_{j\ge 0}\exp\pa{-\eta\gamma 2^{J+j}r}&\le \exp\pa{-\eta\gamma 2^{J}r}\sum_{j\ge 0}\exp\pa{-\eta\gamma 2^{J}jr}\\
&\le  \frac{\exp\pa{-\eta\gamma 2^{J}r}}{1-\exp\pa{-\eta\gamma 2^{J}r}}\le \exp(\Xi(\eta))\exp\pa{-\eta\gamma 2^{J}r},
\end{align*}
which proved the first part of the result. 

In order to prove the second part, it suffices to apply this inequality to $J=0$ and to note that 
\begin{align*}
&\int_{\Theta}\exp\cro{-(2+\eta)\gamma \ell(\theta,\theta')}d\gpi(\theta')\\
&=\int_{\sB}\exp\cro{-(2+\eta)\gamma \ell(\theta,\theta')}d\gpi(\theta')+\int_{\co{\sB}}\exp\cro{-(2+\eta)\gamma \ell(\theta,\theta')}d\gpi(\theta')\\
&\le \gpi(\sB)\pa{1+\frac{e^{-\eta\gamma r}}{1-e^{-\eta \gamma r}}}\le \gpi(\sB)\pa{1+\frac{e^{-\eta}}{1-e^{-\eta}}}=\exp\pa{\Xi(\eta)}\gpi(\sB),
\end{align*}
which yields~\eref{eq-lem-som02}.
\end{proof}

Let us now complete the proof of Theorem~\ref{thmCasDet}.

Let $\eta>0, \gamma>0, J\in \N$ and $r\ge \max\{\ray(\gpi,\theta\et);\gamma^{-1}\}$  to be chosen later on, $\sB=\sB(\theta\et,r)$ and $\sA=\co{\sB}(\theta\et,2^{J}r)$. It follows from \eref{cond-ei} and \eref{def-Ii} that for $i\in\{1,3\}$,
\begin{align}
I_{i,\overline \beta}(\Theta)\ge \int_{\sB}\exp\cro{-\frc_{i}(\overline \beta)\ell\pa{\theta\et,\theta}}d\gpi(\theta)\ge \exp\cro{-\frc_{i}(\overline \beta)r}\gpi(\sB)\label{eq-I13}
\end{align}
and for $i=2$, provided that $\frc_{2}(\overline \beta)\ge (2+\eta)\gamma$, we deduce from Lemma~\ref{lem-som} that 
\begin{align}
I_{2,\overline \beta}(\Theta)&=\int_{\Theta}\exp\cro{-\frc_{2}(\overline \beta)\ell\pa{\theta\et,\theta}}d\gpi(\theta)\le \exp\cro{\Xi(\eta)}\gpi\pa{\sB}.\label{eq-I2}
\end{align}
By arguing in a similar way we obtain that 
\[
I_{1,\beta}(\Theta)\ge \exp\cro{-\frc_{1}(\beta)r}\gpi(\sB)\quad \text{and}\quad I_{2,\beta}(\Theta)\le \exp\cro{\Xi(\eta)}\gpi(\sB),
\]
for $\frc_{2}(\beta)\ge (2+\eta)\gamma$ while for $|\frc_{3}(\beta)|\ge (2+\eta)\gamma$,   
\begin{align*}
I_{3,\beta}(\sA)=\int_{\sA}\exp\cro{-|\frc_{3}(\beta)|\ell(\theta\et,\theta_{3})}d\gpi(\theta_{3})\le \exp\cro{\Xi(\eta)}\gpi\pa{\sB}\exp\pa{-2^{J}\eta\gamma r}.
\end{align*}
It follows thus from Proposition~\ref{prop-det} that for $\min\{\frc_{2}(\beta),\frc_{2}(\overline \beta),|\frc_{3}(\beta)|\}\ge (2+\eta)\gamma$ and $r\ge \max\{\ray(\gpi,\theta\et);\gamma^{-1}\}$ 
\begin{align*}
\cE_{1}&\le I_{0,\overline \beta}\exp\cro{\frc_{1}(\overline \beta)r}\pa{\gpi(\sB)}^{-1}\exp\cro{\Xi(\eta)}\gpi\pa{\sB}\exp\cro{\frc_{3}(\overline \beta)r}\pa{\gpi(\sB)}^{-1}\\
&=\pa{\gpi(\sB)}^{-1}\exp\cro{A_{0}(\overline \beta)+\Xi(\eta)+\pa{\frc_{1}(\overline \beta)+\frc_{3}(\overline \beta)}r}
\end{align*}
and 
\begin{align*}
\cE_{2}(\sA)&\le I_{0,\beta}\exp\cro{\frc_{1}(\beta)r}\pa{\gpi(\sB)}^{-1}\exp\cro{\Xi(\eta)}\gpi\pa{\sB}\exp\cro{\Xi(\eta)}\gpi\pa{\sB}\exp\pa{-2^{J}\eta\gamma r}\\
&=\pa{\gpi(\sB)}\exp\cro{A_{0}(\beta)+2\Xi(\eta)+\frc_{1}(\beta)r-2^{J}\eta\gamma r}.
\end{align*}
For $r\ge \max\{\ray(\gpi,\theta\et);\gamma^{-1}\}$ and $\gamma$ satisfying $(2+\eta)\gamma\le \min\{\frc_{2}(\beta),\frc_{2}(\overline \beta),|\frc_{3}(\beta)|\}$, \eref{eq-020} yields 
\begin{align}
\E_{B}\cro{\gpi_{\bsX}\pa{\co{\sB(\theta\et,2^{J}r)}}}\le 2\exp\pa{\overline A}\label{eq-030}
\end{align}
with 
\[
\overline A=\frac{A_{0}(\overline \beta)+A_{0}(\beta)}{2}+\frac{3\Xi(\eta)}{2}+\frac{\frc_{1}(\overline \beta)+\frc_{1}(\beta)+\frc_{3}(\overline \beta)}{2}r -\frac{2^{J}\eta\gamma r}{2}.
\]
Let us choose the parameters $\eta,\gamma, J$ and $r$ as follows  
\begin{align*}
\eta&=1,\; \quad \text{hence}\quad \frac{3\Xi(1)}{2}= \frac{ -3\log\cro{1-\exp\pa{-1}}}{2}<1;\\
\gamma&=\frac{1}{3}\min\{\frc_{2}(\beta),\frc_{2}(\overline \beta),|\frc_{3}(\beta)|\}\quad \text{hence}\quad (2+\eta)\gamma= \min\{\frc_{2}(\beta),\frc_{2}(\overline \beta),|\frc_{3}(\beta)|\};\\
J&= 2+\PES{\log_{2}\pa{1+\frac{\frc_{1}(\overline \beta)+\frc_{1}(\beta)+\frc_{3}(\overline \beta)}{2\gamma}}}\ge 2;\\
r&=\ray(\gpi,\theta\et)+\frac{A_{0}(\overline \beta)+A_{0}(\beta)+2(1+\xi)}{2\gamma}\ge \max\{\ray(\gpi,\theta\et);\gamma^{-1}\}.
\end{align*}
Setting, $a=\pa{\frc_{1}(\overline \beta)+\frc_{1}(\beta)+\frc_{3}(\overline \beta)}/(2\gamma)$ and $A_{0}=(A_{0}(\overline \beta)+A_{0}(\beta))/2$ we obtain that parameter $J$ satisfies $2^{J-2}\ge 1+a$ and consequently, 
\begin{align*}
\overline A&\le  -\xi +\gamma\pa{\frac{A_{0}+1+\xi}{\gamma}+ar -2^{J-2} r}\le -\xi +\gamma\cro{\frac{A_{0}+1+\xi}{\gamma}-r},
\end{align*}
is not larger than $-\xi$, which, together with  \eref{eq-030}, proves \eref{eq-thmCasDet}. 

\section{Other proofs}\label{sect-proof}

\begin{lem}\label{lem-r-well-defined}
Fix $\theta\in \Theta$. The set 
$$I=\big\{r_0\ge 0: \forall r\ge r_0, \, 0<\gpi\pa{\sB(\theta,2r)}\le \exp\pa{\gamma r}\gpi(\sB(\theta,r))\big\}$$ is an interval of the form 
$[\inf I, \infty)$.
\end{lem}

\begin{proof}
Since $\bigcup_{n\geq 1} \sB(\theta,n) = \Theta$, we have $\lim_n \gpi\pa{\sB(\theta,n)} = 1$,
hence $$\frac{\gpi\pa{\sB(\theta,2r)}}{\gpi\pa{\sB(\theta,r)}}e^{-\gamma r}\CV{r\to +\infty} 0$$
and $I$ is nonempty. Moreover, $I$ satisfies the property that if $r_0\in I$ then every $r\geq r_0$ is in $I$,
thus $I$ is an interval of the form $(\inf I, \infty)$ or $[\inf I, \infty)$.
Let us show that $\inf I\in I$.
We define the function $F:r\mapsto \gpi\pa{\sB(\theta,r)}$ and we denote by $(r_n)_{n\ge 1}$ a nonincreasing sequence in $I$ that converges to $\inf I$.
Since $\gpi$ is a finite measure, $F$ is right-continuous thus 
$\lim_n F(r_n) = F(\inf I)$
and
$\lim_n F(2r_n) = F(2\inf I)$,
hence $F(2\inf I) \le e^{\gamma \inf I} F(\inf I)$.
It remains to show that $F(2\inf I) > 0$.

In the case where $\inf I >0$, we find some $n$ that satisfies $r_n \le 2\inf I$. 
Since $r_n\in I$, this implies 
$0 < F(r_n) \leq F(2\inf I)$.
Finally, we deal with the case where $\inf I =0$.
Since $(0,\infty)\subset I$, it follows by induction that $F(2^{-n}) \ge e^{-\gamma(1-2^{-n})}F(1)$ and by letting $n\to \infty$
we obtain that
$F(0)\ge e^{-\gamma}F(1)>0$.
\end{proof}

\begin{lem}\label{lem-posterior-well-defined}
Under Assumption~\ref{Ass-posterior-well-defined}, the posterior is well-defined.
\end{lem}

\begin{proof}
  Let us fix $\gx\in \gE$ and $\theta \in \Theta$. 
  Since $f\ge 0$, it follows from Assumption~\ref{Ass-posterior-well-defined} that 
  $\lambda \gT(\gx,\theta,\theta')
  \le \lambda K_\gx + \lambda f(\gx,\theta)
  $,
  hence 
  $$0 < \int_{\Theta}\exp\cro{ \lambda \gT(\gx,\theta,\theta')}d\gpi(\theta') 
  \le e^{\lambda K_\gx + \lambda f(\gx,\theta)} \gpi(\Theta)
  < \infty.$$

  Next, we turn to the integral $\int_{\Theta}\gT(\gx,\theta,\theta')\exp\cro{\lambda\gT(\gx,\theta,\theta')}d\gpi(\theta')$.
  Assumption~\ref{Ass-posterior-well-defined} yields the pointwise bound 
  \begin{align*}
    |\gT(\gx,\theta,\theta')|e^{\lambda\gT(\gx,\theta,\theta')}
   &=
   \big|
    \gT_0(\gx,\theta,\theta') + f(\gx,\theta) - f(\gx,\theta')
   \big|
   e^{\lambda (\gT_0(\gx,\theta,\theta') + f(\gx,\theta) - f(\gx,\theta'))}
   \\
   &\le e^{\lambda K_\gx + \lambda f(\gx,\theta)}
   \Big[
    (K_\gx + f(\gx,\theta)) e^{-\lambda f(\gx,\theta')} + f(\gx,\theta')e^{-\lambda f(\gx,\theta')}
   \Big]
   \\
   &\le e^{\lambda K_\gx + \lambda f(\gx,\theta)}
   \Big[
    K_\gx + f(\gx,\theta) + \frac{1}{e\lambda}
   \Big],
  \end{align*}
  where the last bound follows from the inequality $ze^{-z} \leq e^{-1}$ for every $z\geq 0$.

  Finally, we show that $\int_{\Theta}\exp\cro{-\lambda(1-\beta)\gT(\gx,\theta)}d\gpi(\theta) < \infty$.
  We observe that 
  $$\gT(\gx,\theta) = f(\gx,\theta) + \frac{\int_\Theta(\gT_0(\gx,\theta,\theta') - f(\gx,\theta'))e^{\lambda(\gT_0(\gx,\theta,\theta') - f(\gx,\theta'))} d\gpi(\theta')}
    {\int_\Theta e^{\lambda(\gT_0(\gx,\theta,\theta') - f(\gx,\theta'))} d\gpi(\theta')}.$$
  By using the inequality $ze^{z} \geq -e^{-1}$ which holds for every $z\in \R$, we obtain that the numerator is at least $-1/(e\lambda)$. 
  The denominator is at least 
  $e^{-\lambda K_\gx} \int_\Theta e^{-\lambda f(\gx,\theta')} d\gpi(\theta')$
  and consequently
  $$e^{-\lambda(1-\beta)\gT(\gx,\theta)}\le \exp\pa{\frac{(1-\beta)e^{\lambda K_\gx - 1}}{\int_\Theta e^{-\lambda f(\gx,\theta')} d\gpi(\theta')}}
  e^{-\lambda(1-\beta)f(\gx,\theta)}.
  $$
  Since $f\ge 0$, the right-hand side is integrable with respect to $\gpi$.
\end{proof}

\subsection{Proof of Proposition~\ref{prop-lucien}}\label{prf-propL}
Let $\theta\in\Theta$, $r>0$. Since $\frN$ is measurable so is $\cB=\{v\in\R^{D}, \frN(v)\le 1\}$. Under \eref{def-homo} and \eref{eq-Lucien}, we observe that if $\theta'\in\sB(\theta,r)$, $\underline a\frN(\theta'-\theta)\le r$, hence 
$\frN((\underline a/r)^{1/\alpha}(\theta'-\theta))\le 1$. Similarly, if $\frN((\overline  a/r)^{1/\alpha}(\theta'-\theta))\le 1$, then $\overline a\frN(\theta'-\theta)\le r$ and $\theta'\in \sB(\theta,r)$. Setting $r_{1}=(r/\overline a)^{1/\alpha}$ and $r_{2}=(2r/\underline a)^{1/\alpha}$, hence $0<r_{1}\le r_{2}$, we deduce that 
\begin{align*}
B_{1}&=\ac{\theta+r_{1}v,\; v\in\cB}\cap \Theta\subset \sB(\theta,r)\\
B_{2}&=\ac{\theta+r_{2}v,\; v\in\cB}\cap \Theta\supset \sB(\theta,2r).
\end{align*}
Let $A$ be the affine function  defined on $\R^{D}$ by $x\mapsto \theta + (r_{1}/r_{2})(x-\theta)$ and let $\theta' \in B_{2}$. There exists $v\in\cB$ such that $\theta'=\theta+r_{2}v$, hence $A(\theta')=\theta+r_{1}v$ and since $\Theta$ is convex
\[
A(\theta')=\theta + \frac{r_{1}}{r_{2}}(\theta'-\theta)=\pa{1-\frac{r_{1}}{r_{2}}}\theta+\frac{r_{1}}{r_{2}}\theta'\in \Theta.
\]
We deduce that $A(B_{2})\subset B_{1}$ and since the Jacobian of $A$ is $(r_{1}/r_{2})^{D}$, 
\[
\nu(B_{1})\ge \nu(A(B_{2}))=\pa{\frac{r_{1}}{r_{2}}}^{D}\nu(B_{2}). 
\]
Using \eref{eq-Lucien}, we deduce that 
\begin{align*}
\gpi\pa{\sB(\theta,2r)}
\le \overline b\nu\pa{\sB(\theta,2r)}
\le \overline b\nu\pa{B_{2}}
\le \overline b\pa{\frac{r_{2}}{r_{1}}}^{D} \nu\pa{B_{1}}
&\le \frac{\overline b}{\underline b}\pa{\frac{r_{2}}{r_{1}}}^{D} \gpi\pa{\sB(\theta,r)}
\\&=\frac{\overline b}{\underline b}\pa{\frac{2\overline a}{\underline a}}^{D/\alpha} \gpi\pa{\sB(\theta,r)}.
\end{align*}
If $\gpi\pa{\sB(\theta,2r)}=0$, then by iterating we get $\gpi\pa{\sB(\theta,2^j r)}=0$ for each $j\ge 0$,
which contradicts the fact that $\gpi$ is a probability measure.
The claim then follows from the definition~\eref{prop-epsn} of the $\gpi$-complexity.

\subsection{Proof of Proposition~\ref{prop-densite}}\label{prf-prop-densite}
In this proof we assume that $B$ is of the form $\{|I|\le t\}$ for some $t$ for which $\P(B)>0$. We prove that Assumption~\ref{Ass-CasDet} is then satisfied at $\theta\et$ with 
\begin{align*}
\frc_{3}(\lambda,\beta)
&=-\frac{\lambda}{2}\cro{
\frac{1}{2}-\frac{5\beta}{2}
-2\lambda(\beta^{2}+1)\phi\pa{\lambda(1+\beta)}
},\\
\frc_{3}(\lambda,\overline \beta)
&=2\lambda\cro{
\frac{5\overline \beta}{2}-\frac{1}{2}
+2\lambda(\overline \beta^{2}+1)\phi\pa{\lambda(1+\overline \beta)}
}.
\end{align*}
and for $\kappa\in\{\beta,\overline \beta\}$, 
\begin{align*}
\frc_{1}(\lambda,\kappa)
&=2\lambda\cro{
\frac{5}{2}+2\lambda\phi\pa{\lambda(1+\kappa)}
},\\
\frc_{2}(\lambda,\kappa)
&=\frac{\lambda\kappa}{2}\cro{
\frac{1}{2}-2\lambda\kappa\phi\pa{\lambda(1+\kappa)}
}.
\end{align*}
and $A_{0}(\lambda,\kappa,\gP\et,B,\theta\et)=2\lambda (1+\kappa)t-\log\P(B)+\frc_{0}(\lambda,\kappa)nh^{2}(P\et,P_{\theta\et})$ with 
\begin{align*}
\frc_{0}(\lambda,\overline \beta)
&=\lambda\cro{
4+\frac{11}{2}\overline \beta
+2\lambda\pa{4+\overline \beta^{2}}\phi\pa{\lambda(1+\overline \beta)}
},\\
\frc_{0}(\lambda,\beta)
&=\lambda\cro{
\frac{11}{2}-2\beta
+2\lambda\pa{1-2\beta^{2}}\phi\pa{\lambda(1+\beta)}
}.
\end{align*}
We start with the following technical result, which both generalizes and improves over Proposition~3 in Baraud and Birg\'e~\citeyearpar{BarBir2018}.
Note that we defined $H^2$ in \eqref{eq-def-H}.

\begin{prop}\label{prop00}
Let $\mu$ be a $\sigma$-finite measure on $(\cX,\cA)$ and $Q,Q'$ be two finite measures on $(\cX,\cA)$ dominated by $\mu$ with densities $q$ and $q'$ respectively. Then, for every finite measure $R$ on $(\cX,\cA)$,
\begin{align}
\int_{\cX}\psi\pa{\sqrt{\frac{q'}{q}}}dR+\frac{1}{4}\cro{\int_{\cX}q d\mu-\int_{\cX}q'd\mu}
&\le \frac 52 H^{2}\pa{R,Q}-\frac 12 H^{2}\pa{R,Q'}. \label{eq-prop00}
\end{align}
Besides, 
\begin{equation}\label{eq-prop01}
\int_{\cX}\psi^{2}\pa{\sqrt{\frac{q'}{q}}}dR \le  2\cro{H^{2}(R,Q)+H^{2}(R,Q')}.
\end{equation}
\end{prop}

\begin{proof}
We define the $\sigma$-finite measure $\nu = \mu + R$ and we let $f,g$ denote densities such that $\mu = f \cdot \nu$ and $R = g \cdot \nu$.
We also define the functions $q_\nu:x\mapsto q(x)f(x)$ and $q_\nu':x\mapsto q'(x)f(x)$, which satisfy
$Q = q_\nu \cdot \nu$ and $Q' = q_\nu' \cdot \nu$.

We introduce the function $\zeta:\R^3 \to \R$, $(u,v,w)\mapsto \frac 54 (w-u)^2 - \frac 14 (w-v)^2 + \frac 14 (v^2-u^2)$ and
we prove the following inequality:
for every $u,v,w \ge 0$, 
$$\begin{cases}
  &\text{if } u+v >0,\quad \frac{v-u}{v+u}w^2 \le \zeta(u,v,w),
  \\
  &\text{if } u=v=0,\quad \forall t\in (-\infty, 1],\, tw^2 \le \zeta(u,v,w).
\end{cases}$$
First, we fix arbitrary $u,v\ge 0$ such that $u+v>0$ and we define the quadratic function 
$\xi:\R \to \R$, $w\mapsto \zeta(u,v,w) - \frac{v-u}{v+u}w^2 = \frac{2u}{u+v}w^2 + \frac{v-5u}{2}w + u^2$.
In the case where $u=0$, $\xi$ is clearly nonnegative over $[0,\infty)$.
We assume next that $u\neq 0$. The global minimiser of $\xi$ over $\R$ is $w_0=\frac{(5u-v)(u+v)}{8u}$.
If $w_0 \le 0$, then $\forall w\in [0,\infty)$, $\xi(w) \ge \xi(0) = u^2$.
If $w_0 \ge 0$ (i.e., $v\le 5u$), the discriminant of $\xi$ is $\frac{(v-u)^2(v-7u)}{4(u+v)} \le 0$.
Consequently, in every case,  $\xi$ is nonnegative over $[0,\infty)$.
When $u=v=0$, $\zeta(u,v,w) = w^2$ and the claimed inequality holds.

If $x\in \{q_{\nu} + q_{\nu}' > 0\}$, then $f(x) > 0$ and 
$$\psi\pa{\sqrt{\frac{q'(x)}{q(x)}}} 
= \psi\pa{\sqrt{\frac{q_\nu'(x)}{q_\nu(x)}}} 
= \frac{\sqrt{q_\nu'(x)} - \sqrt{q_\nu(x)}}{\sqrt{q_\nu'(x)} + \sqrt{q_\nu(x)}}.$$
On the other hand, if $x\in \{q_{\nu} + q_{\nu}' = 0\} \cap \{f > 0\}$, then $q(x) = q'(x) = 0$ and 
$\psi(\sqrt{q'(x) / q(x)}) = \psi(1) = 0$.
Finally, if $x\in \{q_{\nu} + q_{\nu}' = 0\} \cap \{f = 0\}$, we note that $\psi(\sqrt{q'(x) / q(x)}) \le 1$.
The following inequality is therefore valid for every $x\in \mathcal X$:
$\psi(\sqrt{q'(x) / q(x)})  g(x) \le \zeta\big(\sqrt{q_\nu(x)}, \sqrt{q_\nu'(x)}, \sqrt{g(x)}\big)$, i.e., 
$$\psi\pa{\sqrt{\frac{q'(x)}{q(x)}}} g(x)\le 
\frac 54 \Big(\sqrt{g(x)} - \sqrt{q_\nu(x)}\Big)^2 
- \frac 14 \Big(\sqrt{g(x)} - \sqrt{q_\nu'(x)}\Big)^2 
+ \frac 14 (q_\nu'(x) - q_\nu(x)).$$
By integrating this inequality with respect to the measure $\nu$, we obtain \eqref{eq-prop00}.

We turn to the proof of \eqref{eq-prop01}, which follows a similar blueprint.
With $\zeta:\R^3 \to \R$, $(u,v,w)\mapsto (w-u)^2 + (w-v)^2$,
we show the following inequality:
for every $u,v,w \ge 0$, 
$$\begin{cases}
  &\text{if } u+v >0,\quad \left(\frac{v-u}{v+u}\right)^2 w^2 \le \zeta(u,v,w),
  \\
  &\text{if } u=v=0,\quad \forall t\in [-1,1],\, t^2 w^2 \le \zeta(u,v,w).
\end{cases}$$
Let us fix arbitrary $u,v\ge 0$ such that $u+v>0$.
By the Cauchy--Schwarz inequality for the Euclidean inner product on $\R^2$,
\begin{align*}
  \left(\frac{(v-u)w}{v+u}\right)^2 &= \Big\langle (\frac{v}{v+u}, -\frac{u}{v+u}), (w-u, w-v) \Big\rangle^2
\\&\le \left[\frac{v^2}{(v+u)^2} + \frac{u^2}{(v+u)^2}\right][(w-u)^2 + (w-v)^2]
\le (w-u)^2 + (w-v)^2.
\end{align*}
The case where $u=v=0$ is trivial.

Similar steps as above then yield the inequality
$$\psi^2\pa{\sqrt{\frac{q'(x)}{q(x)}}} g(x)\le 
\Big(\sqrt{g(x)} - \sqrt{q_\nu(x)}\Big)^2 
+ \Big(\sqrt{g(x)} - \sqrt{q_\nu'(x)}\Big)^2 
,$$
which holds for every $x\in \mathcal X$ and \eqref{eq-prop01} follows.

\end{proof}

Next, we establish the following lemma. 
\begin{lem}\label{lem-E-V}
Let $\lambda,\kappa>0$, $\etc{Y}$ be $n$ independent random variables with distributions $\etc{P\et}$ respectively on a measure space $(\cY,\sbY,\nu)$ 
and $P=p\cdot\nu,Q=q\cdot\nu$ and $R=r\cdot \nu$  three probabilities on $(\cY,\sbY)$. We set $P\et=n^{-1}\sum_{i=1}^{n}P_{i}\et$, $a_{0}=5/2$, $a_{1}=1/2$, $a_{2}^{2}=2$ and 
\begin{align}
c_{1}(\lambda,\kappa)&=\lambda\pa{a_{0}+\lambda\phi(\lambda (1+\kappa))a_{2}^{2}},\label{defc1}\\
c_{2}(\lambda,\kappa)&=\lambda\kappa\pa{ a_{1}-\lambda\phi(\lambda (1+\kappa))\kappa a_{2}^{2}}\label{defc2},\\
c_{3}(\lambda,\kappa)&=\lambda\pa{\kappa a_{0}-a_{1}+\lambda\phi( \lambda (1+\kappa))a_{2}^{2}(\kappa^{2}+1)}.\label{defc3}
\end{align}
Then, the random variables 
\[
Z_{i}=\kappa\psi\pa{\sqrt{\frac{q(Y_{i})}{r(Y_{i})}}}-\psi\pa{\sqrt{\frac{p(Y_{i})}{r(Y_{i})}}}\quad \text{for $i\in\{1,\ldots,n\}$}.
\]
satisfy 
\begin{align}
\frac{1}{n}\sum_{i=1}^{n}\E\cro{Z_{i}}&\le (\kappa a_{0}-a_{1})h^{2}(P\et,R)-\kappa a_{1}h^{2}(P\et,Q)+a_{0}h^{2}(P\et,P)\label{eq-E}\\
\frac{1}{n}\sum_{i=1}^{n}\E\cro{Z_{i}^{2}}&\le 2a_{2}^{2}\pa{\kappa^{2}+1}h^{2}(P\et,R)+2\kappa^{2}a_{2}^{2}h^{2}(P\et,Q)+2a_{2}^{2}h^{2}(P\et,P)\label{eq-V}
\end{align}
and 
\begin{align}
\log \E\cro{\exp\pa{\lambda \sum_{i=1}^{n}Z_{i}}}&\le c_{1}(\lambda,\kappa) n  h^{2}(P\et,P)-c_{2}(\lambda,\kappa)n h^{2}(P\et,Q)\label{eq-L}\\
&\hspace{1cm}+ c_{3}(\lambda,\kappa)n h^{2}(P\et,R).\nonumber
\end{align}
When $c_{2}(\lambda,\kappa)\ge 0$, we derive that for every probability measure $\overline P$ on $(\cY,\sY)$, 
\begin{align}
&\log\E\cro{\exp\pa{\lambda\sum_{i=1}^{n}Z_{i}}}\label{eq-L1+}\\
&\le \cro{2c_{1}(\lambda,\kappa)+c_{2}(\lambda,\kappa)+2c_{3}(\lambda,\kappa)}nh^{2}(P\et,\overline P)\nonumber\\
&\hspace{1cm}+2c_{1}(\lambda,\kappa) nh^{2}(\overline P,P)-\frac{c_{2}(\lambda,\kappa)}{2}n h^{2}(\overline P,Q)+2c_{3}(\lambda,\kappa) n h^{2}(\overline P,R)\nonumber
\end{align}
when $c_{3}(\lambda,\kappa)\ge 0$, while for $c_{3}(\lambda,\kappa)<0$ 
\begin{align}
&\log\E\cro{\exp\pa{\lambda\sum_{i=1}^{n}Z_{i}}}\label{eq-L1-}\\
&\le \cro{2c_{1}(\lambda,\kappa)+c_{2}(\lambda,\kappa)+|c_{3}(\lambda,\kappa)|}nh^{2}(P\et,\overline P)\nonumber\\
&\hspace{1cm}+2c_{1}(\lambda,\kappa) nh^{2}(\overline P,P)-\frac{c_{2}(\lambda,\kappa)}{2}n h^{2}(\overline P,Q)-\frac{|c_{3}(\lambda,\kappa)|}{2} n h^{2}(\overline P,R).\nonumber
\end{align}
\end{lem}
\begin{proof}
Let us note that for every integrable function $f$ on $\cY$, 
\begin{align*}
\frac{1}{n}\sum_{i=1}^{n}\E\cro{f(Y_{i})}=\E\cro{f(Y)}
\end{align*}
where $Y$ is a random variable with distribution $P\et$. 
By Proposition~\ref{prop00} and applying successively the previous equality to the functions $\overline \psi:y\mapsto\kappa\psi(\sqrt{q(y)/r(y)})-\psi(\sqrt{p(y)/r(y)})$ and $\overline \psi^{2}$, we derive that 
\begin{align}
\frac{1}{n}\sum_{i=1}^{n}\E\cro{Z_{i}}&=\frac{1}{n}\sum_{i=1}^{n}\E\cro{\overline \psi(Y_{i})}=\E\cro{\overline \psi(Y)}=\kappa\E\cro{\psi\pa{\sqrt{\frac{q(Y)}{r(Y)}}}}-\E\cro{\psi\pa{\sqrt{\frac{p(Y)}{r(Y)}}}}\nonumber\\
&\le \kappa\cro{a_{0}h^{2}(P\et,R)-a_{1}h^{2}(P\et,Q)}-\cro{a_{1}h^{2}(P\et,R)-a_{0}h^{2}(P\et,P)}\nonumber\\
&= a_{0}h^{2}(P\et,P)-\kappa a_{1}h^{2}(P\et,Q)+(\kappa a_{0}-a_{1})h^{2}(P\et,R)\label{eq-lem-EV00}
\end{align}
and 
\begin{align}
\frac{1}{n}\sum_{i=1}^{n}\E\cro{Z_{i}^{2}}&=\E\cro{\overline \psi^{2}(Y)}\le 2\kappa^{2}\E\cro{\psi^{2}\pa{\sqrt{\frac{q(Y)}{r(Y)}}}}+2\E\cro{\psi^{2}\pa{\sqrt{\frac{p(Y)}{r(Y)}}}}\nonumber\\
&\le 2\kappa^{2}a_{2}^{2}\cro{h^{2}(P\et,Q)+h^{2}(P\et,R)}+2a_{2}^{2}\cro{h^{2}(P\et,P)+h^{2}(P\et,R)}\nonumber\\
&=2a_{2}^{2}h^{2}(P\et,P)+2\kappa^{2}a_{2}^{2}h^{2}(P\et,Q)+2a_{2}^{2}\pa{\kappa^{2}+1}h^{2}(P\et,R)\label{eq-lem-EV01}.
\end{align}
This proves \eref{eq-E} and \eref{eq-V}.

Since $\overline \psi$ is not larger than $b=\kappa+1$, it follows from \eref{eq-debase} that for every $i\in\{1,\ldots,n\}$, 
\begin{align*}
\E\cro{\exp\pa{\lambda Z_{i}}}&\le 1+\lambda\E\cro{Z_{i}}+\frac{\lambda^{2}\phi\pa{\lambda b }}{2}\E\cro{Z_{i}^{2}}\le \exp\cro{\lambda\E\cro{Z_{i}}+\frac{\lambda^{2}\phi\pa{\lambda b }}{2}\E\cro{Z_{i}^{2}}}
\end{align*}
and since the $Z_{i}$ are independent, 
\begin{align*}
\log\E\cro{\exp\pa{\lambda\sum_{i=1}^{n}Z_{i}}}&=\sum_{i=1}^{n}\log\E\cro{\exp\pa{\lambda Z_{i}}}\\
&\le \lambda n\pa{\frac{1}{n}\sum_{i=1}^{n}\E\cro{Z_{i}}}+\frac{\lambda^{2}\phi\pa{\lambda b }n}{2}\pa{\frac{1}{n}\sum_{i=1}^{n}\E\cro{Z_{i}^{2}}}.
\end{align*}
By using \eref{eq-E} and \eref{eq-V} we get that 
\begin{align*}
&\log\E\cro{\exp\pa{\lambda\sum_{i=1}^{n}Z_{i}}}\\
&\le \lambda n \cro{ (\kappa a_{0}-a_{1})h^{2}(P\et,R)-\kappa a_{1}h^{2}(P\et,Q)+a_{0}h^{2}(P\et,P)}\\
&\hspace{1cm} +\lambda^{2}\phi\pa{\lambda b }n\cro{a_{2}^{2}\pa{\kappa^{2}+1}h^{2}(P\et,R)+\kappa^{2}a_{2}^{2}h^{2}(P\et,Q)+a_{2}^{2}h^{2}(P\et,P)}\\
&= \lambda\pa{a_{0}+\lambda\phi(\lambda b)a_{2}^{2}}nh^{2}(P\et,P)- \lambda\pa{\kappa a_{1}-\lambda\phi(\lambda b)\kappa^{2}a_{2}^{2}}nh^{2}(P\et,Q)\\
&\hspace{1cm}+ \lambda\pa{\kappa a_{0}-a_{1}+\lambda\phi( \lambda b)a_{2}^{2}(\kappa^{2}+1)}nh^{2}(P\et,R)
\end{align*}
which is \eref{eq-L}. Using the inequalities 
\[
h^{2}(P\et,S)\le 2\cro{h^{2}(P\et,\overline P)+h^{2}(\overline P,S)}\quad \text{and}\quad h^{2}(P\et,S)\ge \frac{1}{2}h^{2}(\overline P,S)-h^{2}(P\et,\overline P)
\]
which holds for every probability $S$, we obtain that  when $c_{2}(\lambda,\kappa)$ and $c_{3}(\lambda,\kappa)$ are nonnegative  
\begin{align*}
\log\E\cro{\exp\pa{\lambda\sum_{i=1}^{n}Z_{i}}}&\le 2c_{1}(\lambda,\kappa) n \cro{h^{2}(P\et,\overline P)+h^{2}(\overline P,P)}\\
&\hspace{0.5cm} -c_{2}(\lambda,\kappa)n \cro{\frac{1}{2}h^{2}(\overline P,Q)-h^{2}(P\et,\overline P)}\\
&\hspace{1.5cm}  + 2c_{3}(\lambda,\kappa)n\cro{h^{2}(P\et,\overline P)+h^{2}(\overline P,R)}
\end{align*}
which gives \eref{eq-L1+}. We obtain \eref{eq-L1-} under the assumption $c_{2}(\lambda,\kappa)\ge 0$, $c_{3}(\lambda,\kappa)<0$ from the inequality 
\begin{align*}
\log\E\cro{\exp\pa{\lambda\sum_{i=1}^{n}Z_{i}}}&\le 2c_{1}(\lambda,\kappa) n \cro{h^{2}(P\et,\overline P)+h^{2}(\overline P,P)}\\
&\hspace{0.5cm} -c_{2}(\lambda,\kappa)n \cro{\frac{1}{2}h^{2}(\overline P,Q)-h^{2}(P\et,\overline P)}\\
&\hspace{1.5cm}  + c_{3}(\lambda,\kappa)n\cro{\frac{1}{2}h^{2}(\overline P,R)-h^{2}(P\et,\overline P)}.
\end{align*}
\end{proof}

We write $\Delta_{1,\kappa}$ in lieu of $\Delta_{\kappa}$ when the test statistic is $\gT_1$. Observe that 
\begin{align*}
\Delta_{1,\kappa}(\bsX,\gtheta)&=\sum_{i=1}^{n}\cro{\kappa\psi\pa{\sqrt{\frac{p_{\theta_{2}}(X_{i})}{p_{\theta_{3}}(X_{i})}}}-\psi\pa{\sqrt{\frac{p_{\theta_{1}}(X_{i})}{p_{\theta_{3}}(X_{i})}}}}\\
&=\sum_{i=1}^{n}\kappa\cro{\psi\pa{\sqrt{\frac{p_{\theta_{2}}(X_{i})}{p_{\theta_{3}}(X_{i})}}}-\psi\pa{\sqrt{\frac{p_{\theta_{2}}(X_{i}\et)}{p_{\theta_{3}}(X_{i}\et)}}}}\\
&\hspace{1cm}-\sum_{i=1}^{n}\cro{\psi\pa{\sqrt{\frac{p_{\theta_{1}}(X_{i})}{p_{\theta_{3}}(X_{i})}}}-\psi\pa{\sqrt{\frac{p_{\theta_{1}}(X_{i}\et)}{p_{\theta_{3}}(X_{i}\et)}}}}\\
&\hspace{2cm}+\sum_{i=1}^{n}\cro{\kappa\psi\pa{\sqrt{\frac{p_{\theta_{2}}(X_{i}\et)}{p_{\theta_{3}}(X_{i}\et)}}}-\psi\pa{\sqrt{\frac{p_{\theta_{1}}(X_{i}\et)}{p_{\theta_{3}}(X_{i}\et)}}}}.
\end{align*}
Since the function $\psi$ is bounded by 1 and the random variables $X_{i}=X_{i}\et$ coincide for $i\not\in I$, we obtain that for every $\gtheta=(\theta_{1},\theta_{2},\theta_{3})\in \Theta^{3}$
\begin{align}
\Delta_{1,\kappa}(\bsX,\gtheta)\le 2(\kappa+1)|I|+\Delta_{1,\kappa}(\bsX\et,\gtheta).\label{eq-b1}
\end{align}
Since on $B$, $|I|\le t$, we deduce from \eref{eq-b1} that for every $\lambda>0$ and $\gtheta=(\theta_{1},\theta_{2},\theta_{3})\in \Theta^{3}$, 
\begin{align}
\log\E_{B}\cro{\exp\pa{\lambda \Delta_{1,\kappa}(\bsX,\gtheta)}}&=-\log\P(B)+\log\E\cro{\exp\pa{\lambda \Delta_{1,\kappa}(\bsX,\gtheta)}\1_{B}}\nonumber\\
&\le -\log\P(B)+\log \E\cro{\exp\cro{\lambda\pa{2(\kappa+1)t+\Delta_{1,\kappa}(\bsX\et,\gtheta)}}\1_{B}}\nonumber\\
&\le -\log\P(B)+2\lambda(\kappa+1)t+\log \E\cro{\exp\pa{\lambda\Delta_{1,\kappa}(\bsX\et,\gtheta)}}.\label{eq-b2}
\end{align}
Since $X_{1}\et,\ldots,X_{n}\et$ are independent, $\Delta_{1,\kappa}(\bsX\et,\gtheta)$ is the sum of the independent bounded random variables 
\[
Z_{i}=\kappa\psi\pa{\sqrt{\frac{p_{\theta_{2}}(X_{i}\et)}{p_{\theta_{3}}(X_{i}\et)}}}-\psi\pa{\sqrt{\frac{p_{\theta_{1}}(X_{i}\et)}{p_{\theta_{3}}(X_{i}\et)}}}\le \kappa+1=b
\]
for $i\in\{1,\ldots,n\}$. Under our assumption that 
\begin{equation}\label{cond-c1c2c3}
\lambda\phi(\lambda (1+\overline \beta))\overline \beta a_{2}^{2}<a_{1}\iff \lambda\overline \beta\phi(\lambda (1+\overline \beta)) < \frac{1}{4}
\end{equation}
we obtain that $a_{1}>\lambda\phi(\lambda (1+\overline \beta))\overline \beta a_{2}^{2}\ge \lambda\phi(\lambda (1+\beta))\beta a_{2}^{2}$, which implies that the constants $c_{2}(\lambda,\kappa)$ defined by \eref{defc2} are  positive for every $\kappa\in \{\beta,\overline \beta\}$. As for $c_{3}(\lambda,\kappa)$ defined by \eref{defc3} we obtain that $c_{3}(\lambda,\overline \beta)\ge \lambda(a_{0}-a_{1})>0$ whilst under the condition 
\[
\frac{5\beta}{4}+\lambda(\beta^{2}+1)\phi( \lambda (1+\beta ))<\frac{1}{4}
\]
we obtain that $\beta a_{0}<a_{1}-\lambda\phi( \lambda (1+\beta ))a_{2}^{2}(\beta^{2}+1)$, hence $c_{3}(\lambda,\beta)$ is negative. 

Applying Lemma~\ref{lem-E-V} with $Y_{i}=X_{i}\et$ for $i\in\{1,\ldots,n\}$, $\overline P=P_{\theta\et}$, $P=P_{\theta_{1}}$, $Q=P_{\theta_{2}}$ and $R=P_{\theta_{3}}$ we deduce from \eref{eq-b2} that
\begin{align*}
&\log\E_{B}\cro{\exp\pa{\lambda \Delta_{1,\overline \beta}(\bsX,\gtheta)}} \\
&\le -\log\P(B)+ 2\lambda(\overline \beta+1)t+\cro{2c_{1}(\lambda,\overline \beta)+c_{2}(\lambda,\overline \beta)+2c_{3}(\lambda,\overline \beta)}nh^{2}(P\et,P_{\theta\et})\nonumber\\
&\hspace{1cm}+2c_{1}(\lambda,\overline \beta) nh^{2}(P_{\theta\et},P_{\theta_{1}})-\frac{c_{2}(\lambda,\overline \beta)}{2}n h^{2}(P_{\theta\et},P_{\theta_{2}})+2c_{3}(\lambda,\overline \beta) n h^{2}(P_{\theta\et},P_{\theta_{3}})
\end{align*}
and 
\begin{align*}
&\log\E_{B}\cro{\exp\pa{\lambda \Delta_{1,\beta}(\bsX,\gtheta)}} \\
&\le -\log\P(B)+2\lambda(\beta+1)t+\cro{2c_{1}(\lambda,\beta)+c_{2}(\lambda, \beta)+|c_{3}(\lambda, \beta)|}nh^{2}(P\et,P_{\theta\et})\nonumber\\
&\hspace{1cm}+2c_{1}(\lambda,\beta) nh^{2}(P_{\theta\et},P_{\theta_{1}})-\frac{c_{2}(\lambda,\beta)}{2}n h^{2}(P_{\theta\et},P_{\theta_{2}})+\frac{c_{3}(\lambda,\beta)}{2} n h^{2}(P_{\theta\et},P_{\theta_{3}}).
\end{align*}
This concludes the proof.

\subsection{Proof of Proposition~\ref{prop-Poisson}}\label{pf-prp-pois}
We prove that Assumption~\ref{Ass-CasDet} is satisfied at $\theta\et$ with $A_{0}(\lambda,\kappa,\gP\et,B,\theta\et)=\lambda (\kappa+1)t-\log \P(B)+\frc_{0}(\lambda,\kappa)\gH^{2}(\gnu\et,\gnu_{\theta\et})$ and the numerical constants $\frc_{i}(\lambda,\kappa)$ for $i\in\{0,1,2,3\}$ and $\kappa\in\{\beta,\overline \beta\}$ given in Section~\ref{prf-prop-densite}.

Let us now turn to the proof of Proposition~\ref{prop-Poisson}. 
Throughout this section, we write $\gT$ for the test statistic $\gT_{2}$ defined by \eref{eq-TPP}. Since $\etc{X\et}$ are $n$ Poisson processes with  intensity measures $\etc{\nu\et}$ respectively,  for $\bsX\et=(\etc{X\et})$ and $\theta,\theta'\in \Theta$
\begin{align*}
\gT(\bsX\et,\theta,\theta')-\E\cro{\gT(\bsX\et,\theta,\theta')}=\sum_{i=1}^{n}\cro{\int_{\cX}\psi\pa{\sqrt{\frac{\theta_{i}'}{\theta_{i}}}}dX_{i}\et-\int_{\cX}\psi\pa{\sqrt{\frac{\theta_{i}'}{\theta_{i}}}}d\nu_{i}\et}
\end{align*}
and for $\gtheta=(\theta_{1},\theta_{2},\theta_{3})\in \Theta^{3}$, we deduce that
\begin{align*}
&\Delta_{\kappa}(\bsX\et,\gtheta)-\E\cro{\Delta_{\kappa}(\bsX\et,\gtheta)}\\
&=\kappa\cro{\gT(\bsX\et,\theta_{3},\theta_{2})-\E\cro{\gT(\bsX\et,\theta_{3},\theta_{2})}}-\cro{\gT(\bsX\et,\theta_{3},\theta_{1})-\E\cro{\gT(\bsX\et,\theta_{3},\theta_{1})}}\\
&=\kappa\sum_{i=1}^{n}\cro{\int_{\cX}\psi\pa{\sqrt{\frac{\theta_{2,i}}{\theta_{3,i}}}}dX_{i}\et-\int_{\cX}\psi\pa{\sqrt{\frac{\theta_{2,i}}{\theta_{3,i}}}}d\nu_{i}\et}\\
&\quad -\sum_{i=1}^{n}\cro{\int_{\cX}\psi\pa{\sqrt{\frac{\theta_{1,i}}{\theta_{3,i}}}}dX_{i}\et-\int_{\cX}\psi\pa{\sqrt{\frac{\theta_{1,i}}{\theta_{3,i}}}}d\nu_{i}\et}.
\end{align*}
Setting, $\overline \psi_{\gtheta,i}=\kappa\psi\pa{\sqrt{\theta_{2,i}/\theta_{3,i}}}-\psi\pa{\sqrt{\theta_{1,i}/\theta_{3,i}}}$ for $i\in\{1,\ldots,n\}$, we derive that 
\begin{align}
\dot{\Delta}_{\kappa}(\bsX\et,\gtheta)&=\Delta_{\kappa}(\bsX\et,\gtheta)-\E\cro{\Delta_{\kappa}(\bsX\et,\gtheta)}=\sum_{i=1}^{n}\int_{\cX}\overline \psi_{\gtheta,i}\pa{dX_{i}\et-d\nu_{i}\et}.\label{def-Deltap}
\end{align}

It follows from Proposition~\ref{prop00}, more precisely \eref{eq-prop00}, with $a_0=5/2$, $a_1=1/2$, $a_2^2=2$,
\begin{align*}
\E\cro{\gT(\bsX\et,\theta_{3},\theta_{2})}&=\sum_{i=1}^{n}\E\cro{\int_{\cX}\psi\pa{\sqrt{\frac{\theta_{2,i}}{\theta_{3,i}}}}dX_{i}\et+\frac{1}{4}\pa{\int_{\cX}\theta_{3,i}d\mu-\int_{\cX}\theta_{2,i}d\mu}}\\
&=\sum_{i=1}^{n}\cro{\int_{\cX}\psi\pa{\sqrt{\frac{\theta_{2,i}}{\theta_{3,i}}}}d\nu_{i}\et+\frac{1}{4}\pa{\int_{\cX}\theta_{3,i}d\mu-\int_{\cX}\theta_{2,i}d\mu}}\\
&\le\ a_0\gH^{2}(\gnu\et,\gnu_{\theta_{3}})-a_1\gH^{2}(\gnu\et,\gnu_{\theta_{2}}).
\end{align*}
Using the fact that $\gT(\bsX\et,\theta,\theta')=-\gT(\bsX\et,\theta',\theta)$ for every $\theta,\theta'\in \Theta$, we also get  
\begin{align*}
-\E\cro{\gT(\bsX\et,\theta_{3},\theta_{1})}&=\E\cro{\gT(\bsX\et,\theta_{1},\theta_{3})}\le a_0 \gH^{2}(\gnu\et,\gnu_{\theta_{1}})- a_1\gH^{2}(\gnu\et,\gnu_{\theta_{3}}).
\end{align*}
Hence 
\begin{align}
\E\cro{\Delta_{\kappa}(\bsX\et,\gtheta)}&=\E\cro{\kappa\gT(\bsX\et,\theta_{3},\theta_{2})-\gT(\bsX\et,\theta_{3},\theta_{1})}\nonumber\\
&\le a_0\gH^{2}\pa{\gnu\et,\gnu_{\theta_{1}}}-a_1\kappa \gH^{2}\pa{\gnu\et,\gnu_{\theta_{2}}}+(a_0\kappa -a_1)\gH^{2}\pa{\gnu\et,\gnu_{\theta_{3}}}.\label{pfPP03}
\end{align}

We note that for every $i\in\{1,\ldots,n\}$
\begin{align*}
\overline \psi_{\gtheta,i}^{2}&=\cro{\kappa\psi\pa{\sqrt{\frac{\theta_{2,i}}{\theta_{3,i}}}}-\psi\pa{\sqrt{\frac{\theta_{1,i}}{\theta_{3,i}}}}}^{2}\le 2\cro{\kappa^{2}\psi^{2}\pa{\sqrt{\frac{\theta_{2,i}}{\theta_{3,i}}}}+\psi^{2}\pa{\sqrt{\frac{\theta_{1,i}}{\theta_{3,i}}}}}
\end{align*}
and it follows thus from inequality \eref{eq-prop01} in Proposition~\ref{prop00} that
\begin{align}
\int_{\cX}\overline \psi_{\gtheta,i}^{2}d\nu_{i}\et&\le 2\kappa^{2}\int_{\cX}\psi^{2}\pa{\sqrt{\frac{\theta_{2,i}}{\theta_{3,i}}}}d\nu_{i}\et+2\int_{\cX}\psi^{2}\pa{\sqrt{\frac{\theta_{1,i}}{\theta_{3,i}}}}d\nu_{i}\et\nonumber\\
&\le 2\kappa^{2}\cro{a_2^2\pa{H^{2}(\nu_{i}\et,\theta_{3,i}\cdot\mu)+H^{2}(\nu_{i}\et,\theta_{2,i}\cdot\mu)}}\nonumber\\
&\quad +2\cro{a_2^2\pa{H^{2}(\nu_{i}\et,\theta_{3,i}\cdot\mu)+H^{2}(\nu_{i}\et,\theta_{1,i}\cdot\mu)}}\nonumber\\
&=2a_2^2\cro{H^{2}(\nu_{i}\et,\nu_{\theta_{1},i})+\kappa^{2}H^{2}(\nu_{i}\et,\nu_{\theta_{2},i})+(\kappa^{2}+1)H^{2}(\nu_{i}\et,\nu_{\theta_{3},i})}.\label{pfPP04}
\end{align}

For every $i\in\{1,\ldots,n\}$, we may apply  \eref{eq-caraPoisson} to the Poisson process $X_{i}\et$ with $f=
\lambda\overline \psi_{\gtheta,i}$ and get 
\begin{align*}
\E\cro{\exp\pa{\int_{\cX}\lambda\overline \psi_{\gtheta,i}\pa{dX_{i}\et-d\nu_{i}\et}}}&=\exp\cro{\int_{\cX}\pa{e^{\lambda\overline \psi_{\gtheta,i}}-\lambda\overline \psi_{\gtheta,i}-1}d\nu_{i}\et}\\
&=\exp\cro{\int_{\cX}\frac{\lambda^{2}\overline \psi_{\gtheta,i}^{2}}{2}\phi\pa{\lambda\overline \psi_{\gtheta,i}}d\nu_{i}\et}\\
&\le \exp\cro{\int_{\cX}\frac{\lambda^{2}\overline \psi_{\gtheta,i}^{2}}{2}\phi\pa{\lambda(\kappa+1)}d\nu_{i}\et}
\end{align*}
since $\phi$ is increasing and $\lambda \overline \psi_{\gtheta}$ is not larger than $\lambda(\kappa+1)$. Using now that the $X_{i}\et$ are independent, we derive from \eref{def-Deltap} that
\begin{align}
\E\cro{\exp\pa{\lambda\dot{\Delta}_{\kappa}(\bsX\et,\gtheta)}}&\le \exp\cro{\frac{\lambda^{2}\phi\pa{\lambda(\kappa+1)}}{2}\sum_{i=1}^{n}\int_{\cX}\overline \psi_{\gtheta,i}^{2}d\nu_{i}\et}\label{pfPP01}
\end{align}
hence, 
\begin{align}
\log\E\cro{\exp\pa{\lambda\Delta_{\kappa}(\bsX\et,\gtheta)}}&=\log\E\cro{\exp\pa{\lambda\E\cro{\Delta_{\kappa}(\bsX\et,\gtheta)}+\lambda\dot{\Delta}_{\kappa}(\bsX\et,\gtheta)}}\nonumber\\
&\le \lambda \E\cro{\Delta_{\kappa}(\bsX\et,\gtheta)}+\frac{\lambda^{2}\phi\pa{\lambda(\kappa+1)}}{2}\sum_{i=1}^{n}\int_{\cX}\overline \psi_{\gtheta,i}^{2}d\nu_{i}\et,\label{pfPP02}
\end{align}
which, with \eref{pfPP03} and \eref{pfPP04}, gives 
\begin{align}
&\log\E\cro{\exp\pa{\lambda\Delta_{\kappa}(\bsX\et,\gtheta)}}\nonumber\\
&\le  \lambda\cro{a_0 \gH^{2}\pa{\gnu\et,\gnu_{\theta_{1}}}-a_1 \kappa \gH^{2}\pa{\gnu\et,\gnu_{\theta_{2}}} + 
(a_0 \kappa - a_1)\gH^{2}\pa{\gnu\et,\gnu_{\theta_{3}}}}\nonumber\\
&\quad + a_2^2 \lambda^{2}\phi\pa{\lambda(\kappa+1)}\cro{\gH^{2}(\gnu\et,\gnu_{\theta_{1}})+\kappa^{2}\gH^{2}(\gnu\et,\gnu_{\theta_{2}})+(\kappa^{2}+1)\gH^{2}(\gnu\et,\gnu_{\theta_{3}})}\nonumber\\
&\le \lambda\pa{a_0 + a_2^2\lambda\phi\pa{\lambda(\kappa+1)}}\gH^{2}(\gnu\et,\gnu_{\theta_{1}})-
\lambda\kappa\pa{a_1-a_2^2\lambda \kappa\phi\pa{\lambda(\kappa+1)}}\gH^{2}\pa{\gnu\et,\gnu_{\theta_{2}}}\label{pfPP05}\\
&\quad +\lambda\pa{a_0 \kappa - a_1+ a_2^2\lambda(\kappa^{2}+1)\phi\pa{\lambda(\kappa+1)}}\gH^{2}\pa{\gnu\et,\gnu_{\theta_{3}}}.\nonumber
\end{align}

Since $|\overline \psi_{\gtheta}|$ is bounded by $\kappa+1$, on the set $B$  we get
\begin{align}
\Delta_{\kappa}(\bsX,\gtheta)-\Delta_{\kappa}(\bsX\et,\gtheta)&=\sum_{i=1}^{n}\int_{\cX}\overline \psi_{\gtheta,i}\pa{dX_{i}-dX_{i}\et}\le (\kappa+1)N\le (\kappa+1)t.\label{pfPP06}
\end{align}
Noting that 
\begin{align*}
\log\sL_{\kappa}(\gtheta,\lambda|B)&=\log\E_{B}\cro{\exp\pa{\lambda \Delta_{\kappa}(\bsX,\gtheta)}}\\
&=\log\E_{B}\cro{\exp\pa{\lambda \Delta_{\kappa}(\bsX,\gtheta)-\lambda \Delta_{\kappa}(\bsX\et,\gtheta)+\lambda \Delta_{\kappa}(\bsX\et,\gtheta)}}\\
&\le \lambda(\kappa+1) t+\log\E_{B}\cro{\exp\pa{\lambda \Delta_{\kappa}(\bsX\et,\gtheta)}}\\
&\le \lambda (\kappa+1)t-\log \P(B)+\log\E\cro{\exp\pa{\lambda \Delta_{\kappa}(\bsX\et,\gtheta)}}
\end{align*}
we derive from \eref{pfPP05}  that
\begin{align*}
\log\sL_{\kappa}(\gtheta,\lambda|B)
&\le \lambda (\kappa+1)t-\log \P(B)\\
&\quad +\lambda\pa{a_0+a_2^2\lambda\phi\pa{\lambda(\kappa+1)}}\gH^{2}(\gnu\et,\gnu_{\theta_{1}})\\
&\hspace{1cm} -\lambda\kappa\pa{a_1-a_2^2\lambda \kappa\phi\pa{\lambda(\kappa+1)}}\gH^{2}\pa{\gnu\et,\gnu_{\theta_{2}}}\\
&\hspace{2cm} +\lambda\pa{a_0 \kappa - a_1+ a_2^2\lambda(\kappa^{2}+1)\phi\pa{\lambda(\kappa+1)}}\gH^{2}\pa{\gnu\et,\gnu_{\theta_{3}}}.
\end{align*}
Under our constraints \eref{eq-constraints1} on $\lambda,\beta$, the following inequalities hold:
\begin{equation}\label{eq-condP}
\begin{cases}
&\dps{
\frac{1}{2}-2\lambda \beta\phi\pa{\lambda(\beta+1)}
>
\frac{1}{2}-2\lambda \overline \beta\phi\pa{\lambda(\overline \beta+1)}
>0
}\\[1.2ex]
&\dps{
\frac{5\overline \beta}{2}-\frac{1}{2}
+2\lambda(\overline \beta^{2}+1)\phi\pa{\lambda(\overline \beta+1)}
>0
}\\[1.2ex]
&\dps{
\frac{5\beta}{2}-\frac{1}{2}
+2\lambda(\beta^{2}+1)\phi\pa{\lambda(\beta+1)}
<0 .
}
\end{cases}
\end{equation}
and we conclude the proof by using the two inequalities below that are consequences of the triangular inequality. For every $\theta\in \Theta$, 
\begin{align*}
\frac{1}{2}\gH^{2}(\gnu_{\theta\et},\gnu_{\theta})-\gH^{2}(\gnu\et,\gnu_{\theta\et})\le \gH^{2}(\gnu\et,\gnu_{\theta})\le 2\gH^{2}(\gnu\et,\gnu_{\theta\et})+2\gH^{2}(\gnu_{\theta\et},\gnu_{\theta}).
\end{align*}

\subsection{Proof of Proposition~\ref{prop-piex}}
The proof relies on the following lemmas that are proven in Sections~\ref{pf-lem4} and \ref{pf-lem5} respectively. 
\begin{lem}\label{UnitSphere}
Let $s\et$ be in the unit sphere of a Hilbert space $\H$ equipped with its norm $\norm{\cdot}$, let $V\ne \{0\}$ be a closed linear subspace of $\H$ and let $S$ denote the unit sphere of $V$. Then 
\[
\inf_{s\in S}\norm{s\et-s}^{2}\le 2\inf_{v\in V}\norm{s\et-v}^{2}.
\]
\end{lem}

\begin{lem}\label{lem-sphere}
Let $\nu_{D}$ be the uniform distribution on the sphere $\S_{D}$ of $\R^{D}$ where $D\ge 1$.  
For each $D\ge 2$ and every $u\in\S_{D}$ and $t\in [0,2]$, 
\begin{equation}\label{eq-cap0}
 \nu_{D}\pa{\ac{v\in\S_{D},\; \ab{u-v}\le t}}=\frac{\Gamma(2\alpha)}{\cro{\Gamma\pa{\alpha}}^{2}}\int_{0}^{t^{2}/4}z^{\alpha-1}(1-z)^{\alpha-1}dz
\end{equation}
with  $\alpha=(D-1)/2$. In particular, for $D=2$ and $t\in [0,2]$
\begin{equation}\label{eq-cap1}
\frac{t}{\pi}\le  \nu_{2}\pa{\ac{v\in\S_{2},\; \ab{u-v}\le t}}=\frac{2}{\pi}\arcsin\pa{\frac{t}{2}}\le \frac{t}{2}
\end{equation}
and for every $D\ge 3$ and $t\in [0,\sqrt{2}]$, 
\begin{equation}\label{eq-cap2}
\sqrt{\frac{2}{\pi D}}\pa{\frac{t}{\sqrt{2}}}^{D-1}\le  \nu_{D}\pa{\ac{v\in\S_{D},\; \ab{u-v}\le t}}\le \frac{t^{D-1}}{\sqrt{2\pi(D-1)}}.
\end{equation}

For every $D\ge 1$, every $u\in \S_D$, and every
$t>0$,
\begin{equation}\label{eq-cap3}
  \frac{2^{-D/2}}{\sqrt{\pi D}}\,(t\wedge 1)^{D-1}
\le 
\nu_D(\{v\in \S_D:\ |u-v|\le t\}).
\end{equation}
\end{lem}

Note that Inequality \eref{eq-prop-approx} is a direct consequence of  \eref{H-st} and Lemma~\ref{UnitSphere}.

Now let $(\rho,a,s)\in\Gamma_{J,m}$. Then $s$ can be written as $\sum_{j=1}^{D_{m}}s_{j}\phi_{j,m}\in S_{m}$. If $(\varrho,b,t)\in\Gamma_{J,m}$ satisfy 
\[
\ab{\rho-\varrho}^{2}\le \frac{2r}{9n},\; \ab{a-b}^{2}\le\frac{2r}{9n\rho^{2}}\quad \text{and}\quad \norm{s-t}^{2}\le \frac{2r}{9n\rho^{2}},
\]
it follows from \eref{eq-H-E} that 
\begin{align*}
\ell(\theta_{\rho,a,s},\theta_{\varrho,b,t})=\gH^{2}\pa{\gnu_{\rho,a,s},\gnu_{\varrho,b,t}}&\le \frac 32 n\cro{\rho^{2}\frac{2r}{9n\rho^{2}}+\rho^{2}\frac{2r}{9n\rho^{2}}+\frac{2r}{9n}}= r.
\end{align*}

Let $R,\bsY_{\!\!J},\bsZ_{m}$ be the random variables defined in Section~\ref{sect-ex}. 
It follows from the definition of $\gpi_{J,m}$ and the identity 
\[
\norm{s-\sum_{j=1}^{D_{m}}Z_{j,m}\phi_{j,m}}^{2}=\ab{\gs-\bsZ_{m}}^{2}\quad \text{with}\quad \gs=(s_{1},\ldots,s_{D_{m}})
\]
that
\begin{align*}
\gpi_{J,m}\pa{\sB(\theta_{\rho,a,s},r)}&\ge \P\cro{\ab{R-\rho}^{2}\le \frac{2r}{9n},\; \ab{a-\bsY_{\! \!J}}^{2}\le\frac{2r}{9n\rho^{2}},\ab{\gs-\bsZ_{m}}^{2}\le \frac{2r}{9n\rho^{2}}}\\
&= \P\cro{\ab{R-\rho}\le \sqrt{\frac{2r}{9n}}}\P\cro{\ab{\bsY_{\!\! J}-a}\le\sqrt{\frac{2r}{9n\rho^{2}}}}\P\cro{\ab{\bsZ_{m}-\gs}\le \sqrt{\frac{2r}{9n\rho^{2}}}}.
\end{align*}

Let us first note that for every $\rho,\zeta>0$, 
\begin{align*}
\P\cro{\ab{R-\rho}\le \zeta}\ge \frac{2}{\pi}\int_{\rho}^{\rho+\zeta}\frac{dx}{1+x^{2}}\ge \frac{2}{\pi}\int_{\rho}^{\rho+\zeta}\frac{dx}{(1+x)^{2}}=2\cro{\pi(1+\rho)(1+(1+\rho)\zeta^{-1})}^{-1}.
\end{align*}
This yields the inequality 
\begin{equation}\label{eq-uni1}
\P\cro{\ab{R-\rho}\le \sqrt{\frac{2r}{9n}}}\ge \frac{2}{\pi}\cro{(1+\rho)\pa{1+(1+\rho)\sqrt{\frac{9n}{2r}}}}^{-1}.
\end{equation}

Using \eqref{eq-cap3}, we obtain, with
$\varepsilon=\sqrt{\frac{2r}{9n\rho^2}}$
that
\begin{align}
\P\cro{\ab{\bsY_{\!\! J}-a}\le\varepsilon}
&\ge
\frac{2^{-|J|/2}}{\sqrt{\pi |J|}}
(\varepsilon\wedge1)^{|J|-1},\label{eq-uniS-new}\\
\P\cro{\ab{\bsZ_m-\gs}\le\varepsilon}
&\ge
\frac{2^{-D_m/2}}{\sqrt{\pi D_m}}
(\varepsilon\wedge1)^{D_m-1}.\label{eq-uniD-new}
\end{align}
Putting together inequalities \eref{eq-uni1}, \eref{eq-uniS-new} and
\eref{eq-uniD-new}, and setting \(d=|J|+D_m\), we get
\begin{align*}
-\log\cro{\gpi_{J,m}\pa{\sB(\theta_{\rho,a,s},r)}}
&\le
\log\pa{\frac{\pi^{2}(1+\rho)\sqrt{|J|D_m}}{2}}
+\frac d2\log 2
+\log\pa{1+(1+\rho)\sqrt{\frac{9n}{2r}}}\\
&\quad
+\pa{\frac d2-1}
\log\pa{\frac{9n\rho^2}{2r}\vee1}.
\end{align*}
Since
$
\frac{9n\rho^2}{2r}
\le
(1+\rho)^2\frac{9n}{2r}
$
we have
$$
\log\pa{\frac{9n\rho^2}{2r}\vee1}
\le
2\log\pa{1+(1+\rho)\sqrt{\frac{9n}{2r}}},
$$
and therefore
\begin{align*}
-\log\cro{\gpi_{J,m}\pa{\sB(\theta_{\rho,a,s},r)}}
&\le
d\log\pa{1+(1+\rho)\sqrt{\frac{9n}{2r}}}
+\frac d2\log 2
+\log\pa{\frac{\pi^{2}(1+\rho)\sqrt{|J|D_m}}{2}}\\
&=
d\log\pa{\sqrt2\pa{1+(1+\rho)\sqrt{\frac{9n}{2r}}}}
+\log\pa{\frac{\pi^{2}(1+\rho)\sqrt{|J|D_m}}{2}}.
\end{align*}

Let us define
\begin{align*}
\overline r
&=
\frac{|J|+D_m}{\gamma}
\log\pa{
e + 3(1+\rho)\sqrt{\frac{\gamma n}{(|J|+D_m)\vee L_{J,m}}}
}
+\frac{1}{\gamma}
\log\pa{\frac{\pi^{2}(1+\rho)\sqrt{|J|D_m}}{2}}
+\frac{L_{J,m}}{\gamma}.
\end{align*}
If $r\ge \overline r$, then $r\ge \frac{(|J|+D_m)\vee L_{J,m}}{\gamma}$
and
$$
(1+\rho)\sqrt{\frac{9n}{2r}}
\le
\frac{3}{\sqrt 2}(1+\rho)\sqrt{\frac{\gamma n}{(|J|+D_m)\vee L_{J,m}}}.
$$
Writing $x=(1+\rho)\sqrt{\frac{\gamma n}{(|J|+D_m)\vee L_{J,m}}}$, we get
$$
\sqrt 2 \pa{1+(1+\rho)\sqrt{\frac{9n}{2r}}}
\le \sqrt 2 \pa{1+\frac{3}{\sqrt 2}x}
=\sqrt 2 + 3x
\le e + 3x.
$$
Therefore
$$
-\log\cro{\gpi_{J,m}\pa{\sB(\theta_{\rho,a,s},r)}}
\le
\gamma\overline r-L_{J,m}.
$$
Consequently,
\begin{align*}
-\log\cro{\gpi\pa{\sB(\theta_{\rho,a,s},r)}}
&=-\log\cro{\sum_{(J',m')\in\cJ\times \cM}e^{-L_{J',m'}}\gpi_{J',m'}\pa{\sB(\theta_{\rho,a,s},r)}}
\\
&\le
-\log\cro{\gpi_{J,m}\pa{\sB(\theta_{\rho,a,s},r)}}+L_{J,m}
\le
\gamma\overline r\le \gamma r.
\end{align*}
This proves the claimed bound on \(\overline r(\gpi,\theta_{\rho,a,s})\).

\subsection{Proof of Lemma~\ref{UnitSphere}}\label{pf-lem4}
Let $s'$ be the orthogonal projection of $s\et$ onto $V$. We define $s$ as $s=s'/\norm{s'}$ when $\norm{s'}>0$ and $s$ is an arbitrary unit vector in $V$ otherwise. Then, 
\begin{align*}
\norm{s\et-s}^{2}=2\pa{1-\norm{s'}}=\frac{2\pa{\norm{s\et}^{2}-\norm{s'}^{2}}}{1+\norm{s'}}=\frac{2\norm{s\et-s'}^{2}}{1+\norm{s'}}\le 2\norm{s\et-s'}^{2}.
\end{align*}

\subsection{Proof of Lemma~\ref{lem-sphere}}\label{pf-lem5}
Fix some $D\geq 2$ and
let $\widehat v$ be a random variable with distribution $ \nu_{D}$. The random variable $\scal{u}{\widehat v}$ admits a density with respect to the Lebesgue measure given by 
\[
q_{D}(x)=c_{D}\pa{1-x^{2}}^{(D-3)/2}\1_{[-1,1]}\quad \text{with}\quad c_{D}=2^{2-D}\frac{\Gamma(D-1)}{\cro{\Gamma\pa{(D-1)/2}}^{2}}=2^{1-2\alpha}\frac{\Gamma(2\alpha)}{\cro{\Gamma\pa{\alpha}}^{2}}.
\]
Since $\ab{u}=\ab{\widehat v}=1$, 
\[
\ab{u-\widehat v}\le t\iff 2\pa{1-\scal{u}{\widehat v}}\le t^{2}\iff \scal{u}{\widehat v}\ge 1-\frac{t^{2}}{2}
\]
and we obtain that for $t\in [0,2]$
\begin{align}
 \nu_{D}\pa{\ac{v\in\S_{D},\; \ab{u-v}\le t}}&=\P\cro{\scal{u}{\widehat v}\ge 1-\frac{t^{2}}{2}}=c_{D}\int_{1-t^{2}/2}^{1}\pa{1-x^{2}}^{\alpha-1}dx.\label{lem1-eq0}
\end{align}
For $D=2$, we get that for every $t\in [0,2]$
\begin{align*}
 \nu_{D}\pa{\ac{v\in\S_{D},\; \ab{u-v}\le t}}&=\frac{1}{\pi}\int_{1-t^{2}/2}^{1}\frac{dx}{\sqrt{1-x^{2}}}=\frac{1}{\pi}\cro{\arcsin(x)}_{1-t^{2}/2}^{1}\\
&=\frac{1}{\pi}\cro{\frac{\pi}{2}-\arcsin\pa{1-\frac{t^{2}}{2}}}=\frac{1}{\pi}\arccos\pa{1-\frac{t^{2}}{2}}\\
&=\frac{2}{\pi}\arcsin\pa{\frac{t}{2}}.
\end{align*}
In the general case, by doing the change of variables $x=1-2z$ we get 
\begin{align}
\int_{1-t^{2}/2}^{1}\pa{1-x^{2}}^{\alpha-1}dx&=\int_{0}^{t^{2}/4}\pa{1-\pa{1-2z}^{2}}^{\alpha-1}2dz=2\int_{0}^{t^{2}/4}\cro{4z(1-z)}^{\alpha-1}dz\nonumber\\
&=2^{2\alpha-1}\int_{0}^{t^{2}/4}z^{\alpha-1}(1-z)^{\alpha-1}dz\label{Lem1-eq1}
\end{align}
which with \eref{lem1-eq0} gives \eref{eq-cap0}. 

Furthermore, the following inequalities (see Wendel~\citeyearpar{MR29448})
\[
z^{1-s}\le \frac{\Gamma(z+1)}{\Gamma(z+s)}\le \pa{z+s}^{1-s}
\]
hold for every $z>0$ and $s\in (0,1)$. Applying them with $s=1/2$ we obtain that 
\begin{align}
\frac{\alpha}{\sqrt{\pi (\alpha+1/2)}}\le c_{D}=\frac{\alpha\Gamma(\alpha+1/2)}{\Gamma\pa{\alpha+1}\sqrt{\pi}}\le \sqrt{\frac{\alpha}{\pi}}\label{lem1-eq2}.
\end{align}

Let us now assume that $t\in [0,\sqrt{2}]$ and $D\ge 3$. Then $\alpha\ge 1$ and every $z\in [0,t^{2}/4]$ satisfies $1\ge 1-z\ge 1-t^{2}/4\ge 1/2$. We deduce that 
\begin{align}
\frac{1}{\alpha}\pa{\frac{t}{\sqrt{2}}}^{2\alpha}\le 2^{2\alpha-1}\int_{0}^{t^{2}/4}z^{\alpha-1}(1-z)^{\alpha-1}dz\le \frac{t^{2\alpha}}{2\alpha}.\label{lem1-eq4}
\end{align}
Inequalities \eref{lem1-eq0}, \eref{Lem1-eq1}, \eref{lem1-eq2} and \eref{lem1-eq4} yield for $D\ge 3$ and $t\in [0,\sqrt{2}]$
\begin{align*}
\frac{1}{\sqrt{\pi(\alpha+1/2)}}\pa{\frac{t}{\sqrt{2}}}^{2\alpha}\le  \nu_{D}\pa{\ac{v\in\S_{D},\; \ab{u-v}\le t}}\le \frac{t^{2\alpha}}{2\sqrt{\pi \alpha}},
\end{align*}
which is \eref{eq-cap2}.

We turn to the proof of \eref{eq-cap3}. For $D=1$, the sphere is $\{-1,1\}$ and for every $t>0$ we have
$
\nu_1(\{v\in \S_1: |u-v|\le t\})\ge \frac12
\ge \frac{1}{\sqrt{2\pi}}
$.
For $D=2$ and $0<t\le 1$, \eref{eq-cap1} yields
$$
\nu_2(\{v\in \S_2: |u-v|\le t\})
\ge \frac{t}{\pi}
\ge \frac{2^{-1}}{\sqrt{2\pi}}t .
$$
For $t>1$, monotonicity in $t$ yields the same bound with $t$ replaced by $1$.
For $D\ge 3$ and $0 < t \le 1$, \eref{eq-cap2} implies that
\[
\nu_D(\{v\in \S_D: |u-v| \le t\})
\ge
\sqrt{\frac{2}{\pi D}}\left(\frac{t}{\sqrt2}\right)^{D-1}
=
\frac{2^{1-D/2}}{\sqrt{\pi D}} t^{D-1}
\ge
\frac{2^{-D/2}}{\sqrt{\pi D}} t^{D-1}.
\]
For $t>1$, monotonicity again gives the bound with $t \wedge 1$.

\bibliographystyle{apalike}

\end{document}